\documentclass[reqno,12pt,letterpaper]{amsart}
\usepackage{amsmath,amssymb,amsthm,graphicx,mathrsfs,url,bbm}
\usepackage[usenames,dvipsnames]{color}
\usepackage[colorlinks=true,linkcolor=Red,citecolor=Green]{hyperref}
\usepackage[all]{xy}

\def\?[#1]{\textbf{[#1]}\marginpar{\Large{\textbf{??}}}}
\newtheorem{prop}{Proposition}
\newtheorem{thm}[prop]{Theorem}
\newtheorem{defi}[prop]{Definition}

\newtheorem{lem}[prop]{Lemma}

\newtheorem{rem}[prop]{Remark}

\newtheorem{conj}[prop]{Conjecture}
\numberwithin{equation}{section}
\numberwithin{prop}{section}

\renewcommand{\Re}{\mathop{\rm Re}\nolimits}
\renewcommand{\Im}{\mathop{\rm Im}\nolimits}

\DeclareMathOperator{\Op}{Op}

\DeclareMathOperator{\sgn}{sgn}
\DeclareMathOperator{\supp}{supp}

\DeclareMathOperator{\Tr}{{\rm tr}}

\DeclareMathOperator{\dist}{dist}
\DeclareMathOperator{\nbd}{nbd}
\DeclareMathOperator{\el}{ell}
\DeclareMathOperator{\comp}{comp}
\DeclareMathOperator{\rank}{rank}
\DeclareMathOperator{\WF}{WF}
\DeclareMathOperator{\range}{{\rm Ran}}
\DeclareMathOperator{\Res}{{\rm Res}}
\def\indic{\mathbbm 1}

\newcommand{\Ecal}{{\mathcal E}}
\newcommand{\Fcal}{{\mathcal F}}

\newcommand{\Hcal}{{\mathcal H}}

\newcommand{\Ocal}{{\mathcal O}}
\newcommand{\Pcal}{{\mathcal P}}
\newcommand{\Ucal}{{\mathcal U}}

\newcommand{\Mcal}{{\mathcal M}}
\newcommand{\Scal}{{\mathcal S}}
\newcommand{\Dcal}{{\mathcal D}}
\newcommand{\Xbf}{{\mathbf X}}
\newcommand{\ubf}{{\mathbf u}}
\newcommand{\RR}{{\mathbb R}}
\newcommand{\CC}{{\mathbb C}}

\newcommand{\ZZ}{{\mathbb Z}}

\begin{document}
\title[Pollicott--Ruelle resonances]{Counting Pollicott--Ruelle resonances \\ for Axiom A flows}

\author{Long Jin}
\email{jinlong@mail.tsinghua.edu.cn}
\address{Mathematical Sciences Center, Tsinghua University, Beijing, China \& Beijing Institute of Mathematical Sciences and Applications, Beijing, China}

\author{Zhongkai Tao}
\email{ztao@math.berkeley.edu}
\address{Department of Mathematics, Evans Hall, University of California,
Berkeley, CA 94720, USA}

\begin{abstract} 
In this paper, we count the number of Pollicott--Ruelle resonances for open hyperbolic systems and Axiom A flows. In particular, we prove polynomial upper bounds and sublinear lower bounds on the number of resonances with modulus less than $r$ in strips for open hyperbolic systems and Axiom A flows with a transversality condition.
\end{abstract}

\maketitle
\section{Introduction}
\label{s:intro}
For Axiom A flows $\varphi^t=e^{tX}$ on a compact manifold $M$, Pollicott \cite{pollicott} and Ruelle \cite{ruelle1,ruelle2} introduced the concept of resonances based on the study of correlation function:
$$\rho_{f,g}(t)=\int_M(f\circ\varphi^t)gdx.$$
Pollicott--Ruelle resonances are the poles of the Fourier--Laplace transform $\hat{\rho}_{f,g}(\lambda)$ of $\rho_{f,g}(t)$, initially only defined as a holomorphic function in a half plane $\Im\lambda\gg 1$, and meromorphically continued to a larger region in the complex plane.

For the special cases of Anosov flows, the Pollicott--Ruelle resonances can be equivalently described as the poles of the meromorphic continuation of the resolvent operator $R(\lambda)=(P-\lambda)^{-1}$, where $P=-iX$ and $X$ is the generator of the flow $\varphi^t$, see Faure--Sj\"{o}strand \cite{FS}. The method of anisotropic Banach spaces used there originates in the works of  Baladi--Tsujii \cite{baladi}, Blank--Keller--Liverani \cite{BKL}, Butterley--Liverani \cite{butterley}, Gou\"ezel--Liverani \cite{gouezel}, and Liverani \cite{liverani1,liverani2}. In this way, we can define the set of Pollicott--Ruelle resonances $\Res(P)$ in the whole complex plane $\mathbb{C}$ for $C^\infty$ Anosov flows. This is further generalized to open hyperbolic systems by Dyatlov--Guillarmou \cite{open}, Morse--Smale flows by Dang--Rivi\`ere \cite{morsesmale1,morsesmale2} and Axiom A flows with strong transversality condition by Meddane \cite{meddane}. 

For Axiom A flows on a compact manifold, the Ruelle zeta function encodes the information of closed orbits. 
\begin{defi}
Let $\varphi^t=e^{tX}$ be an Axiom A flow on a compact manifold, the Ruelle zeta function is defined as
\begin{equation}
\label{zeta}
\zeta_{R}(\lambda)=\prod_{\gamma^\#}(1-e^{-\lambda T_{\gamma^\#}}),\quad \Re\lambda\gg 1.
\end{equation}
Here the sum is taken over all primitive closed orbits $\gamma^\#$ of the flow, and $T_{\gamma}^\#$ denotes the length of the primitive closed orbit $\gamma^\#$.
\end{defi}
Smale \cite{smale} conjectured that certain related zeta functions have meromorphic continuation to the whole complex plane. For Ruelle zeta functions of $C^\infty$ Anosov flows, this was first proved by Giulietti--Liverani--Pollicott \cite{GLP} and an alternative microlocal proof was given by Dyatlov--Zworski \cite{zeta}. Later, Dyatlov--Guillarmou \cite{axioma} proved the case of Axiom A flows based on their previous work \cite{open} extending the microlocal methods of \cite{zeta} to open hyperbolic systems. In particular, the zeros and poles of the Ruelle zeta function are related to Pollicott--Ruelle resonances for lifts of the flow to certain vector bundles.

In the current paper, we study the distribution of Pollicott--Ruelle resonances, counted with multiplicity, for Axiom A flows satisfying the strong transversality condition which roughly says the weakly stable and unstable manifolds can only intersect transversally (see Section \ref{s:axiomabasic} for the more detailed definition). This condition is natural in the theory of structural stability, e.g. $C^1$ structural stability is equivalent to Axiom A with strong transversal condition (see Robbin  \cite{robbin}, Robinson \cite{robinson2}, Ma\~{n}\'{e} \cite{mane} for diffeomorphisms and Robinson \cite{robinson1}, Liao \cite{liao}, Hu \cite{hu}, Wen \cite{wen}, Hayashi \cite{hayashi} for flows.)

\begin{thm}
\label{t:axioma}
Let $n=\dim M$. Suppose $\varphi^t=e^{tX}$ is an Axiom A flow on $M$ satisfying the strong transversality condition,
then for any fixed $\beta>0$, there exists $C>0$ such that for any $E>0$,
\begin{equation}
\label{e:axioma-upper}
\#\Res(P)\cap\{\lambda:|\lambda|\leq E,\Im\lambda>-\beta\}\leq CE^{n+1}+C.
\end{equation}
If in addition, $\varphi^t$ has at least one closed orbit, then for any $\delta\in(0,1)$ there exist $\beta>0$ and $C>0$ depending on $\delta$ such that 
\begin{equation}
\label{e:axioma-lower}
\#\Res(P)\cap\{\lambda:|\lambda|\leq E,\Im\lambda>-\beta\}\geq \frac{1}{C}E^\delta-C.
\end{equation}
\end{thm}

Following Dyatlov--Guillarmou \cite{axioma}, we can use the spectral decomposition for Axiom A flows to reduce to the situation of open hyperbolic systems, see also Meddane \cite{meddanethesis}. In particular, we have the following results.
\begin{thm}
\label{t:upper}
Let $(\Ucal, X)$ be an open hyperbolic system (see Section \ref{s:open} for definition) with $\dim\Ucal=n$, $P=-iX$ and $\Res(P)$ the set of Pollicott--Ruelle resonance of $(\Ucal, X)$ counted with multiplicity, then for any constant $\beta>0$, 
\begin{equation}
\label{e:upper-1}
    \# \Res(P)\cap\{\lambda\in\mathbb{C}\,:\, |\Re\lambda-E|\leq 1, {\rm Im}\, \lambda\geq -\beta\}= \mathcal{O}(E^n).
\end{equation}
In particular, there exists $C>0$ such that for any $E>0$,
\begin{equation}
\label{e:upper-2}
    \# \Res(P)\cap\{\lambda\in\mathbb{C}\,:\, |\lambda|<E, \Im\lambda>-\beta\}\leq CE^{n+1}+C.
\end{equation}
\end{thm}

\begin{thm}
\label{t:lower}
With the same assumption as Theorem \ref{t:upper}, for any $\delta \in (0,1)$, there exist $\beta>0$ and $C>0$ depending on $\delta$ such that for any $E>0$,
\begin{equation}
\label{e:lower}
\# \Res(P)\cap\{\lambda\in\mathbb{C}\,:\, |\lambda|<E, \Im\lambda>-\beta\}\geq \frac{1}{C}E^\delta-C.
\end{equation}
In particular, there are infinitely many resonances.
\end{thm}


Here we only state the scalar case for simplicity, the general situation of the lift to vector bundles is discussed in Section \ref{s:bundle}.

\subsection{Previous works and the method of the proof}
There are many works on the distrubition of Pollicott--Ruelle resonances in special cases. For example, Dang--Rivi\`{e}ve \cite{morsesmale1,morsesmale2} gave detailed descriptions for Pollicott--Ruelle resonances for Morse--Smale flows. Following the suggestion from Malo J\'{e}z\'{e}quel, we discuss Morse--Smale flows using trace formula \eqref{e:localtrace} later in Section \ref{s:morsesmale}, see also \cite{meddanethesis} for a similar approach using dynamical determinants directly. Dyatlov--Faure--Guillarmou \cite{DFG} built the relation between Pollicott--Ruelle resonances for geodesic flows on closed hyperbolic manifolds and eigenvalues of the Laplace operator. 

For general Anosov flows, polynomial upper bound on the number of resonances was first proved by Faure--Sj\"{o}strand \cite{FS}. Faure--Tsujii \cite{fwl} later gave a better bound with fractal exponent depending on the H\"{o}lder regularity for the strong stable and unstable distribution. In the case of contact Anosov flows, Datchev--Dyatlov--Zworski \cite{DDZ} gave the optimal polynomial bound. Moreover, for contact Anosov flows with certain pinching conditions on minimal and maximal expansion rates, Faure--Tsujii \cite{band0,band1,band} proved a band structure for the distribution of resonances as well as the Weyl asymptotics in a band, see also Ceki\'{c}--Guillarmou \cite{band3d} for the first band on 3-dimensional contact Anosov flows. In general, motivated by the work of Sj\"{o}strand \cite{sjostrand} on quantum resonances, we expect that a fractal Weyl law holds with the exponent related to the dimension of certain trapped set in phase space. For further results on quantum resonances and the analogy between quantum resonances and Pollicott--Ruelle resonances, we refer the reader to the review paper \cite{zworskireview}.

The sublinear lower bound \eqref{e:lower} was first studied in the case of Anosov flows by Jin--Zworski \cite{local} (with a gap filled by the authors \cite{flattrace}), but they only show a weaker lower bound. We improve their method here to get the stronger lower bound \eqref{e:lower}.  Similar results for scattering resonances were obtained  by Guillop\'{e}--Zworski \cite{GZ} for convex cocompact hyperbolic surfaces and Petkov \cite{petkov} for several strictly convex obstacles. On how to use the trace formula to conclude the lower bound, we refer to \cite{SZ1, SZ2, Schenck}. A general lower bound matching an appropriate upper bound seems very difficult to obtain. As in \cite{local}, a key ingredient is the following local trace formula which we establish in the setting of Theorem \ref{t:upper}.
\begin{thm}
\label{t:localtrace}
Let $(\Ucal, X)$ be an open hyperbolic system. For any $A>0$ there exists a distribution $F_A\in \Scal'(\RR)$ supported in $[0,\infty)$ such that
\begin{equation}
\label{e:localtrace}
    \sum\limits_{\mu\in\Res(P),\Im(\mu)>-A}e^{-i\mu t}+F_A(t)=\sum\limits_{\gamma }\frac{T_\gamma^\#\delta(t-T_\gamma)}{|\det(I-\mathcal{P}_\gamma)|},\quad t>0
\end{equation}
in $\mathcal{D}'((0,\infty))$, where the sum on the right hand side is taken over all closed geodesics $\gamma$ with $\mathcal{P}_\gamma$ the corresponding Poincar\'e map, $T_\gamma$ the period and $T_\gamma^\#$ the primitive period. Moreover, we have the following estimate on the Fourier--Laplace transform of $F_A$:
\begin{align*}
    |\widehat{F}_A(\lambda)|=\Ocal_{A,\epsilon}(\langle \lambda\rangle^{2n+1+\epsilon}),\quad \Im \lambda<A-\epsilon
\end{align*}
for any $\epsilon>0$.
\end{thm}

\subsection{Organization of the paper}
The paper is organized as follows: In Section \ref{s:prelim}, we review dynamics of open hyperbolic systems and semiclassical microlocal analysis, mostly from \cite{open}. In Section \ref{s:upper}, we prove Theorem \ref{t:upper} on polynomial upper bound for the number of Pollicott--Ruelle resonances for open hyperbolic systems. In Section \ref{s:lower}, we prove Theorem \ref{t:localtrace} on local trace formula and then deduce Theorem \ref{t:lower} from Theorem \ref{t:localtrace}. In Section \ref{s:axioma} we use spectral decomposition to prove Theorem \ref{t:axioma} using Theorem \ref{t:upper} and \ref{t:lower}. Finally in Section \ref{s:example}, we discuss some special examples, including suspension of Axiom A maps, Morse--Smale flows and geodesic flows on hyperbolic manifolds.

\subsection*{Acknowledgement}
The authors would like to thank Semyon Dyatlov and Maciej Zworski for encouragement and numerous discussions, Gabriel Rivi\`ere for informing us the work of Antoine Meddane, and Vesselin Petkov for interest in the project. We also thank Malo J\'{e}z\'{e}quel for many suggestions, including using trace formula to obtain detailed descriptions for Pollicott--Ruelle resonance. Long Jin is supported by National Key R\&D Program of China 2022YFA100740. Zhongkai Tao gratefully acknowledges partial support under the NSF grant DMS-1952939 and the Simons Targeted Grant Award No. 896630, and the support of Morningside Center of Mathematics during his visit.

\section{Setup and Preliminaries}
\label{s:prelim}

\subsection{Dynamics of open hyperbolic systems}
\label{s:open}
We first review the definition and some basic properties of open hyperbolic systems from \cite{open}, and we refer to \cite{open} for details. Some related examples are geodesic flows on asymptotically hyperbolic manifolds \cite{mazzeomelrose} with negative curvature and billiard flows with several convex obstacles \cite{DSW,CP}.

An open hyperbolic system consists of a compact $C^\infty$ manifold with boundary $\overline{\Ucal}$ and a $C^\infty$ nonvanishing vector field $X$ on $\overline{\Ucal}$. We denote by $\varphi^t=e^{tX}$ the flow generated by $X$ and we make the following assumptions.

\begin{itemize}
    \item[(C)] (Convexity condition) the boundary $\partial\Ucal$ is strictly convex with respect to $X$ in the following sense: let $\rho\in C^\infty(\overline{\Ucal})$ be a boundary defining function, that is, $\rho>0$ on $\Ucal$, $\rho=0$ and $d\rho\neq0$ on $\partial \Ucal$, then
\begin{equation}
\label{c:convexity}
x\in\partial\Ucal, X\rho(x)=0\quad\quad\Longrightarrow\quad\quad X^2\rho(x)<0.
\end{equation}
This condition does not depend on the choice of $\rho$.
\end{itemize}
Following \cite[Lemma 1.1]{open}, we will always fix 
\begin{itemize}
\item an embedding of $\overline{\Ucal}$ into a compact manifold $\mathcal{M}$ without boundary;
\item an extension of the boundary defining function $\rho$ to $\mathcal{M}$ such that $\rho<0$ on $\mathcal{M}\setminus\overline{\Ucal}$;
\item an arbitrary extension $X_1$ of $X$ to $\mathcal{M}$;
\item constants $c_0,\varepsilon_0>0$ such that on $\{|\rho|\leq 2\varepsilon_0\}$, $X_1^2\rho<c_0(X_1\rho)^2$.
\end{itemize}
From this, we can fix an extension of $X$ to $\mathcal{M}$, which will still be denoted by $X$, obtained by
$$X=\psi(\rho)X_1,\,\,\psi\in C^\infty(\RR),\,\,
\psi(s) =1 \text { for } s\geq 0, \,\, \sgn \psi(s)=\sgn(s+\epsilon_0), \,\, \psi'(-\epsilon_0)>0,$$
such that $\mathcal{U}$ is convex with respect to $X$ in the following sense:
\begin{equation}
\label{c:convexity-flow}
x,\varphi^T(x)\in\Ucal, T\geq0\quad\quad\Longrightarrow\quad\quad \forall t\in[0,T],\varphi^t(x)\in\Ucal.
\end{equation}
We can also assume that the same convexity condition holds for $\overline{\Ucal}$ with respect to $X$.

Now let 
$$\Gamma_\pm=\bigcap_{\pm t\geq 0}\varphi^t(\overline{\Ucal})\subset\overline{\Ucal}$$
be the incoming and outgoing tail, and $K=\Gamma_+\cap\Gamma_-$ be the trapped set, then $K\subset\mathcal{U}$ is a compact invariant subset under $\varphi^t$ and we assume that
    
\begin{itemize}
    \item[(H)] (Hyperbolicity condition) $K$ is hyperbolic in the sense that for every $x\in K$, we can decompose $T_x\mathcal{M}$ into direct sum of flow/stable/unstable subspaces
$$T_x\mathcal{M}=E_0(x)\oplus E_s(x)\oplus E_u(x),$$
where $E_0,E_s,E_u$ are continuous in $x$ and invariant under $d\varphi^t$,$E_0(x)=\mathbb{R}X(x)$, and given any Riemannian metric on $T\mathcal{M}$, there exists constants $C,\theta>0$ such that
\begin{equation*}
\begin{split}
|d\varphi^t(x)v|_{\varphi^t(x)}\leq Ce^{-\theta|t|}|v|_x,&\;\;\; v\in E_u(x), \ \ t<0\\
|d\varphi^t(x)v|_{\varphi^t(x)}\leq Ce^{-\theta|t|}|v|_x,&\;\;\; v\in E_s(x),
\ \ t>0,
\end{split}
\end{equation*}
\end{itemize}

We list several properties for the flow $\varphi^t$ and the trapped set $K$
\begin{prop}
\label{p:property-K}
\begin{itemize}
    \item[(1)] \cite[Lemma 1.3]{open} For $x\in\Gamma_\pm$, we have $\varphi^t(x)\to K$ uniformly as $t\to \mp \infty$.
    \item[(2)] \cite[Lemma 1.4]{open} Let $V$ be a neighbourhood of $K$, then there exists $T>0$ such that if $\varphi^T(x),\varphi^{-T}(x)\in\overline{\mathcal{U}}$, then $x\in V$.
\end{itemize}
\end{prop}

Another relevant set on $\mathcal{M}$ is the whole interaction region with $\mathcal{U}$, which is denoted by 
$$\Sigma=\Sigma_+\cup\Sigma_-,\quad \Sigma_\pm:=\bigcup_{\pm t\geq0}\varphi^t(\mathcal{U}).$$
In particular, by the choice of $\rho,\varepsilon_0$, we have (\cite[Lemma 1.7]{open})
$$\Sigma\subset\{\rho>-\varepsilon_0\},\quad
\overline{\Sigma}\cap\{\rho=-\varepsilon_0\}\cap\{X_1\rho=0\}=\varnothing.$$

We also record the following lemma concerning the propagation of a compact subset:
\begin{lem}\label{l:convexity}
Let $V$ be a compact subset of $\Ucal$, then
$$\tilde{V}:=\overline{\bigcup_{t\geq0}\varphi^t(V)}
\cap\overline{\bigcup_{t\leq0}\varphi^t(V)}$$
is also a compact subset of $\Ucal$.
\end{lem}
\begin{proof}
Since $X$ vanishes on $\rho^{-1}(-\varepsilon_0)$, we see that $\tilde{V}\subset\{\rho\geq-\varepsilon_0\}$. There exists $\varepsilon_1\in(0,\varepsilon_0)$ such that $V\cap\{\rho<\varepsilon_1\}=\varnothing$, we claim that 
$$\tilde{V}\cap\{\rho<\varepsilon_1\}=\varnothing$$
and thus $\tilde{V}$ is also a compact subset of $\Ucal$.

To see this, we use the convexity condition: by \cite[Lemma 1.6]{open}, for any $x\in\{-\varepsilon_0\leq\rho<\varepsilon_1\}$, if $X_1\rho(x)\leq 0$,  there exists $T\geq0$, such that $e^{TX_1}(x)\in\{\rho=-\varepsilon_0\}$ and $\rho(e^{tX_1}(x))\in[-\varepsilon_0,\varepsilon_1)$ for any $t\in(0,T]$. Therefore for any $t\geq0$, $\varphi^t(x)\in\{-\varepsilon_0<\rho<\varepsilon_1\}$. The same applies to $-X_1$ and thus if $X_1\rho(x)\geq0$, then for any $t\leq0$, $\varphi^t(x)\in\{-\varepsilon_0<\rho<\varepsilon_1\}$. In both cases, we see that $x\not\in\tilde{V}$ and thus $\tilde{V}\cap\{\rho<\varepsilon_1\}=\varnothing$.
\end{proof}

\subsection{Dynamics on the cotangent bundle}
\label{s:cotangent}
Now we discuss the lift to the cotangent bundle. Let $p(x,\xi)=\langle X(x),\xi\rangle$, then the flow $\varphi^t$ can be lifted to the contangent space $T^\ast\Mcal$ as $e^{tH_p}(x,\xi)=(\varphi^t(x),(d\varphi^t(x))^{-T}\xi)$. For each $x\in K$, there is a dual flow/stable/unstable decomposition
\begin{align*}
    T_x^\ast\mathcal{M}=E^*_0(x)\oplus E^*_s(x)\oplus E^*_u(x),
\end{align*}
where $E^*_0(x)$ is the annihilator of $E_s(x)\oplus E_u(x)$, and similarly for $E^\ast_s$ and $E^\ast_u$. The hyperbolicity condition (H) implies a similar condition for this dual decomposition:
\begin{equation*}
\begin{split}
|(d\varphi^t(x))^{-T}\xi|\leq Ce^{-\theta|t|}|\xi|,&\;\;\; \xi\in E_s^\ast(x), \ \ t\geq 0;\\
|(d\varphi^t(x))^{-T}\xi|\leq Ce^{-\theta|t|}|\xi|,&\;\;\; \xi\in E_u^\ast(x),
\ \ t\leq0.
\end{split}
\end{equation*}
$E_s^\ast$ and $E_u^\ast$ can be further extended to vector bundles $E^\ast_\pm\subset T^\ast_{\Gamma_\pm}\mathcal{M}$ over $\Gamma_\pm$ such that
\begin{itemize}
    \item $E^\ast_+|_K=E^\ast_u$, $E^\ast_-|_K=E^\ast_s$, and $E^\ast_\pm$ depends continuously on $x\in\Gamma_\pm$.
    \item $E^\ast_\pm$ are invariant under $\varphi^t$ and $\langle X,
    \eta\rangle=0 $ for $\eta\in E^\ast_{\pm}$.
    \item There exists constants $\widetilde{C},\widetilde{\theta}>0$ such that for all $x\in \Gamma_\pm$ and $\xi\in E^*_\pm(x)$, as $t\to \mp\infty$, 
    \begin{align*}
        |(d\varphi^t(x))^{-T}\xi|\leq \widetilde{C}e^{-\widetilde{\theta}|t|}|\xi|.
    \end{align*}
    \item If $x\in\Gamma_\pm$, $\xi\in T^*_x\mathcal{M}\setminus E^*_\pm(x)$ such that $p(x,\xi)=0$, then as $t\to \mp\infty$, 
    \begin{align*}
        |(d\varphi^t(x))^{-T}\xi |\to\infty,\quad \frac{(d\varphi^t(x))^{-T}\xi}{|(d\varphi^t(x))^{-T}\xi|}\to E^*_\mp|_K.
    \end{align*}
\end{itemize}

We will need to localize to the energy shell $\{p=1\}$. In order to study the dynamics near $\{p=1\}$, we need the following lemma analogous to \cite[Lemma 1.10]{open}.
\begin{lem}
There exist continuous vector bundles $E_{\pm 0}^*$ over $\Gamma_{\pm}$ such that $E_{+0}^*|_K=E_u^*\oplus E_0^*$ and $E_{-0}^*|_K=E_s^*\oplus E_0^*$. Moreover, \begin{itemize}
    \item $E_{\pm 0}^*$ is invraint under $\varphi^t$.
    \item For all $x\in \Gamma_\pm$ and $\xi\in E^*_{\pm 0}$, we have
    $(d\varphi^t(x))^{-T}\xi\to E_0^*$ as $t\to \mp\infty$.
    \item If $x\in\Gamma_\pm$, $\xi\in T^*_x\mathcal{M}\setminus E^*_{\pm0}(x)$, then as $t\to \mp\infty$, 
    \begin{align*}
        |(d\varphi^t(x))^{-T}\xi |\to\infty,\quad \frac{(d\varphi^t(x))^{-T}\xi}{|(d\varphi^t(x))^{-T}\xi|}\to E^*_\mp|_K.
    \end{align*}
\end{itemize}
\end{lem}
\begin{proof}
We recall that $\Gamma_\pm$ has a lamination by unstable/stable manifolds (see for example \cite[\S 3.3]{nozw})
\begin{equation*}
    \Gamma_\pm=\bigcup\limits_{y\in K} W^{u/s}(y).
\end{equation*}
We simply define $E_{\pm 0}^*(x)\subset T_{x}^*\Mcal$ to be the annihilator of $T_{x} W^{u/s}(y)$ for $x\in\Gamma_\pm$. The flow invariance is clear. We fix a Riemannian metric $\tilde{g}$ on $\Mcal$ and define $\tau_{x\to y}:T_x^*\Mcal\to T_y^*\Mcal$ by parallel transport.
For each $y\in K$, define the projections
\begin{align*}
    \pi_s(y):T_y^*\Mcal\to E_s^*(y),\quad \pi_u(y):T_y^*\Mcal\to E_u^*(y).
\end{align*}

For $x\in\Gamma_\pm$ and $\xi\in E_{\pm 0}^*(x)$, since $(d\varphi^t(x))^{-T}\xi\to E_{\pm0}^*|_K$ as $t\to\mp\infty$, we know for $y_t\in K$ with $\dist(\varphi^t(x),y_t)\to 0$,
$\pi_{s/u} \tau_{\varphi^t(x)\to y_t} (d\varphi^t)^{-T}\xi\to 0$ as $t\to\mp\infty$.
By hyperbolicity, we also have
$\pi_{u/s} \tau_{\varphi^t(x)\to y_t} (d\varphi^t)^{-T}\xi\to 0$ as $t\to\mp\infty$.
Thus $(d\varphi^t)^{-T}\xi\to E_0^*$.

For $x\in\Gamma_\pm$ and $\xi\in T_x^*\Mcal\setminus E_{\pm0}^*(x)$, there exists $v\in T_xW^{u/s}(y)$ such that $\langle \xi,v\rangle\neq 0$. Then
\begin{equation*}
    \langle (d\varphi^t)^{-T}\xi,(d\varphi^t)v\rangle=\langle \xi,(d\varphi^t)^{-1}(d\varphi^t)v\rangle=\langle \xi,v\rangle,\quad |(d\varphi^t )v|\leq Ce^{-\theta|t|},\quad t\to\mp\infty
\end{equation*}
which implies $|(d\varphi^t)^{-T}\xi|\to \infty$ and consequently 
$$\frac{(d\varphi^t(x))^{-T}\xi}{|(d\varphi^t(x))^{-T}\xi|}\to E^*_\mp|_K$$
by hyperbolicity.
\end{proof}
\subsection{Preliminaries on semiclassical and microlocal analysis}
We give a brief review of semiclassical analysis and refer to \cite{semibook}, \cite[Appendix E]{resbook} for more detailed discussions, and \cite[\S 2]{open} for materials related to the microlocal analysis for open hyperbolic systems.

We will always work on a compact manifold $\mathcal{M}$ without boundary. 
For a symbol on a manifold, we can define quantization by a partition of unity argument. This gives the (non-canonical) quantization map
\begin{align*}
    {\rm Op}_h: S^k_h(T^*\mathcal{M})\to \Psi^k_h(\mathcal{M})
\end{align*}
from the symbol class to the class of pseudodifferential operators. There is a canonical isomorphism called symbol map
\begin{align*}
    \sigma_h:\Psi^k_h(\mathcal{M})/h\Psi^{k-1}_h(\mathcal{M})\to S^k_h(T^*\mathcal{M})/hS^{k-1}_h(T^*\mathcal{M}).
\end{align*}

The semiclassical Sobolev space $H^s_h(\mathcal{M})$ is defined by the norm
\begin{align*}
    \|u\|_{H^s_h}=\|\langle hD\rangle^s u\|_{L^2}.
\end{align*}
An $h$-dependent distribution $u\in\mathcal{D}'(\mathcal{M})$ is called $h$-tempered if it lies in $\mathcal{H}^{-N}_{h,loc}$ for some $N$ with norm $\mathcal{O}(h^{-N})$.
\begin{itemize}
    \item For $a\in S^k_h(T^*\mathcal{M})$, we define its essential support ${\rm esssupp}\, a\subset \overline{T}^*\mathcal{M}$ such that $(x_0,\xi_0)\notin {\rm esssupp}\, a$ if and only if there exists a neighbourhood $W$ of $(x_0,\xi_0)$, with
\begin{align*}
    |\partial_x^\alpha\partial_\xi^\beta a(x,\xi)|\leq C h^N\langle \xi\rangle^{-N},\quad (x,\xi)\in W\cap T^*\mathcal{M}.
\end{align*} 
\item Suppose $A={\rm Op}_h(a)+\mathcal{O}(h^
\infty)_{\Psi^\infty}\in \Psi^k_h(\mathcal{M})$, then the wavefront set ${\rm WF}_h(A)\subset \overline{T}^*\mathcal{M}$ is defined to be ${\rm esssupp}\, a$. The elliptic set is defined as 
\begin{align*}
    {\rm ell}_h(A)=\{(x_0,\xi_0)\in\overline{T}^*\mathcal{M}\,:\, (\langle \xi\rangle^{-k}\sigma_h(A))(x_0,\xi_0)\neq 0\}.
\end{align*}
\item Given an $h$-tempered distribution $u\in \mathcal{D}'_h(\mathcal{M})$, we define its semiclassical wavefront set ${\rm WF}_h(u)\subset \overline{T}^*\mathcal{M}$ such that $(x_0,\xi_0)\notin {\rm WF}_h(u)$ if and only if there exists a neighbourhood $U$ of $(x_0,\xi_0)$ such that for any properly supported $A\in\Psi^k_h(\mathcal{M})$ with ${\rm WF}_h(A)\subset
U$,
\begin{align*}
    Au=\mathcal{O}(h^\infty)_{C^\infty}.
\end{align*}
\item Let $B(h):C_c^\infty (\mathcal{M}_2)\to\mathcal{D}'(\mathcal{M}_1)$ be an $h$-tempered family of operators, then ${\rm WF}'_h(B)\subset \overline{T}^*(\mathcal{M}_1\times\mathcal{M}_2)$ is defined as 
\begin{align*}
    {\rm WF}'_h(B)=\{(x,\xi,y,\eta)\,:\, (x,\xi,y,-\eta)\in {\rm WF}_h(\mathcal{K}_B)\}
\end{align*}
where $\mathcal{K}_B$ is the Schwartz kernel of $B$.
\end{itemize}


The wavefront set condition allows us to define the flat trace of an operator, see \cite[\S 8.2]{hormander}. 
\begin{prop}
Let $\mathcal{M}$ be a manifold. If $B:C^\infty(\mathcal{M}) \to\mathcal{D}'(\mathcal{M})$ satisfies
\begin{align*}
    {\rm WF}'(B)\cap \Delta(T^*\mathcal{M})=\varnothing,
\end{align*}
then $\iota^*K_B\in\mathcal{D}'(\mathcal{M})$ is well-defined and we can define
\begin{align*}
    {\rm tr}^\flat B=\int_\mathcal{M} \iota^*K_B(x) dx=\langle \iota^*K_B,1\rangle .
\end{align*}
\end{prop}

Recall \cite[Lemma 3.2]{flattrace}, for a family of operators, we have 
\begin{lem}
\label{l:semiwf-flattrace}
If $P(h):C^\infty(\Mcal)\to \mathcal{D}'(\Mcal)$ is $h$-tempered and satisfies
\begin{itemize}
    \item ${\rm WF}_h'(P(h))\cap \Delta(S^*\Mcal)=\varnothing$;
    \item $\|AP(h)B\|_{L^2\to L^2}=\mathcal{O}(h^{-m})$ for $A,B\in \Psi^{\rm comp}_h(\mathcal{M})$;
\end{itemize}
then ${\rm tr}^\flat (P(h))$ is well-defined with
\begin{align*}
    {\rm tr}^\flat (P(h))=\mathcal{O}(h^{-2n-m}).
\end{align*}
\end{lem}


We briefly recall some semiclassical estimates in \cite{open}. 
\begin{defi} 
Let $P\in\Psi^1_h(\mathcal{M})$, $L$ be a closed set in $\overline{T}^*\mathcal{M}$, we say 
\begin{align*}
    P\lesssim -h, \text{  on  } H^s_h \text{  microlocally near  }L
\end{align*}
if there exist operators $Y_1\in\Psi^s_h(\mathcal{M})$, $Y_2\in \Psi^{-s}_h(\mathcal{M})$, and $Z\in\Psi^0_h(\mathcal{M})$, with
\begin{align*}
    Y_1Y_2=1+\mathcal{O}(h^\infty) \text{  near  }L, \quad L\subset {\rm ell}_h(Z)
\end{align*}
and for each $N$ and small enough $h$ we have
\begin{align*}
    {\rm Im}\,\langle Y_1PY_2 u,u\rangle_{L^2}\leq -h\|Zu\|_{L^2}+\mathcal{O}(h^\infty)\|u\|_{H^{-N}_h},\quad u\in H^\frac{1}{2}_h.
\end{align*}
\end{defi}

\begin{prop}
\label{p:semiestimate}
Let $\mathcal{M}$ be a compact manifold. 
\begin{itemize}
    \item[(1)] Let $A\in\Psi^0_h(\mathcal{M})$. If $P\in\Psi^k_h(\mathcal{M})$ is elliptic on ${\rm WF}_h(A)$, then
    \begin{align*}
        \|Au\|_{H^m_h}\leq C\|Pu\|_{H^{m-k}_h}+\mathcal{O}(h^\infty)\|u\|_{H^{-N}_h}.
    \end{align*}
      \item[(2)] If $P\in\Psi^{1}_h(\mathcal{M})$, $Q\in \Psi^0_h(\mathcal{M})$ such that 
    \begin{align*}
        {\rm Im}\, \sigma_h(P)\leq 0 \text{ near } L,\quad {\rm Re}\, \sigma_h(Q)>0\text{ on }L,
    \end{align*}
    then ${\rm Im}(P-iMhQ)\lesssim -h$ on $H^s_h$ near $L$ for sufficienitly large constant $M$.
\end{itemize}
\end{prop}
\begin{proof}
(1) is the standard elliptic estimate. (2) is a stronger version of \cite[Lemma 2.5]{open} with essentially the same proof.
\end{proof}

Now we recall the propagation estimates from \cite{open}:
\begin{defi}
Let $p\in S^1(T^*\mathcal{M})$ be a real-valued symbol.
Let $V,W$ be two open sets in  $\overline{T}^*\mathcal{M}$. We define the set ${\rm Con}_p(V,W)$ that is controlled by $V$ inside of $W$, to be the set of points $(x,\xi)\in\overline{T}^*\mathcal{M}$ such that there exists $T>0$ such that 
\begin{align*}
    e^{-TH_p}(x,\xi)\in V,\quad e^{-tH_p}(x,\xi)\in W, \text{ for } 0\leq t\leq T.
\end{align*}
\end{defi}
\begin{prop}
\label{p:propa}
\begin{itemize}
    \item[(1)]\cite[Lemma 2.6]{open} Let $P\in\Psi^{1}_h(\mathcal{M})$ satisfy $p={\rm Re}\,\sigma_h(P)\in S^1(T^*\mathcal{M},\mathbb{R})$, $L\subset T^*\mathcal{M}$ is invariant under $e^{tH_p}$ and $\Im \sigma_h(P)\leq 0$ near $L$. 
    \begin{itemize}
        \item If there exist $c,\gamma>0$ such that
    \begin{align*}
        \frac{|e^{tH_p}(x,\xi)|}{|\xi|}\geq ce^{\gamma|t|},\quad (x,\xi)\in L, t\leq 0,
    \end{align*}
    then there exists $s_0$ such that ${\rm Im}\, P\lesssim -h$ near $L$ on $H^s_h$ for $s>s_0$.
    \item If there exist $c,\gamma>0$ such that
    \begin{align*}
        \frac{|e^{tH_p}(x,\xi)|}{|\xi|}\geq ce^{\gamma|t|},\quad (x,\xi)\in L, t\geq 0,
    \end{align*}
    then there exists $s_0$ such that ${\rm Im}\, P\lesssim -h$ near $L$ on $H^s_h$ for $s<s_0$.
    \end{itemize} 
    \item[(2)]\cite[Lemma 2.7]{open} Assume $\sigma_h(P)=p-iq$ with $p\in S^1(T^*\mathcal{M},\mathbb{R})$, $L\subset T^*\mathcal{M}$ is compact and invariant under $e^{tH_p}$. Let $A, B, B_1\in \Psi^0_h(\mathcal{M})$, $s\in\mathbb{R}$ such that
    \begin{align*}
        {\rm WF}_h(A)\subset {\rm ell}_h(B_1), L\subset {\rm ell}_h(A), L\cap {\rm WF}_h(B)=\varnothing\\
        q\geq 0 \text{ near } {\rm WF}_h(B_1), {\rm Im}\, P\lesssim -h \text{ on } H^s_h \text{ near } L.
    \end{align*}
    Let
    \begin{align*}
        \Omega=\{\langle\xi\rangle^{-1}p=0\}\setminus {\rm Con}_p({\rm ell}_h(B);{\rm ell}_h(B_1))
    \end{align*}
    and assume for $(x,\xi)\in \Omega\cap {\rm WF}_h(A)$, locally uniformly
    \begin{align*}
        e^{tH_p}(x,\xi)\to L,\quad t\to -\infty;\quad  e^{tH_p}(x,\xi)\in {\rm WF}_h(B_1),\quad t\leq 0.
    \end{align*}
    Then for any $N$ and sufficiently small $h>0$, we have
    \begin{align*}
        \|Au\|_{H^s_h}\leq C\|Bu\|_{H^s_h}+Ch^{-1}\|B_1Pu\|_{H^s_h}+\mathcal{O}(h^\infty)\|u\|_{H^{-N}_h}.
    \end{align*}
\end{itemize}
\end{prop}

\subsection{Spectral theory of open hyperbolic systems}
\label{s:spectral}
In \cite{open}, Dyatlov--Guillarmou prove that for open hyperbolic systems $(\overline{\Ucal},X)$, the restricted resolvent for $P=-iX$
$$R(\lambda):=\indic_\Ucal(P-\lambda)^{-1}\indic_\Ucal: C_c^\infty(\Ucal)\to\mathcal{D}'(\Ucal)$$
extends meromorphically from a half-plane $\Im\lambda>C_0$ to the whole complex plane $\mathbb{C}$ with poles of finite rank. These poles are called the Pollicott--Ruelle resonances for $P$ and the set of resonances is denoted by $\Res(P)$. Our notation here is only different from \cite{open} by considering $P=-iX$ instead of $X$ itself. The key is the construction of the microlocal weight function and anisotropic Sobolev spaces
\begin{prop}
\label{c:weight}
There exists $m\in C^\infty(S^*\mathcal{M};\mathbb{R})$ such that
\begin{itemize}
    \item $m=1$ in a neighbourhood of $\kappa(E^*_-)\supset \kappa(E^*_s)$;
    \item $m=-1$ in a neighbourhood of $\kappa(E^*_+)\supset \kappa(E^*_u)$;
    \item $H_pm\leq 0$ in a neighbourhood of $\{p=0\}\cap S^*\mathcal{M}$;
    \item $\supp m\subset \{\rho>-2\varepsilon_0\}$ and $\supp m\cap \{\rho=-\varepsilon_0\}\cap \{X_1\rho=0\}=\varnothing$.
\end{itemize}
\end{prop}
Let
\begin{align*}
    \widetilde{m}(x,\xi)=(1-\chi_m(x,\xi))m(x,\xi)\log |\xi|
\end{align*}
where $\chi_m\in C^\infty(T^*\mathcal{M};[0,1])$ is equal to $1$ near the zero section and supported in some ball $\{|\xi|<c_0\}$ for some constant $c_0>0$.
Define the operator $G=G(h)\in \Psi^{0+}_h(\mathcal{M})$ such that
\begin{itemize}
\item $\sigma_h(G)=\widetilde{m}$;
\item $\WF_h(G)\subset\{\rho>-2\varepsilon_0\}$;
\item $\WF_h(G)\cap\{\rho=-\varepsilon_0\}\cap \{X_1\rho=0\}=\varnothing$.
\end{itemize}
The anisotropic Sobolev space is given by $\mathcal{H}^s_h=\exp(-sG(h))L^2(\mathcal{M})$ with norm
$\|u\|_{\mathcal{H}^s_h}=\|e^{sG(h)}u\|_{L^2(\mathcal{M})}$.
Take complex absorbing potentials
$Q_\infty \in \Psi^1_h$, $q_1\in C^\infty$ supported outside $\overline{\mathcal{U}}$ such that
\begin{itemize}
    \item $\indic_{\Sigma'}Q_\infty=Q_\infty\indic_{\Sigma'}=0$ for some neighborhood $\Sigma'$ of $\overline{\Sigma}$;
    \item $\{\rho\leq -2\varepsilon_0\}\cup (\{\rho=-\varepsilon_0\}\cap \{X_1\rho=0\})\subset {\rm ell}_h(Q_\infty)$;
    \item $\WF_h(G)\cap \WF_h(Q_\infty)=\varnothing$;    
    \item $\supp q_1\cap \overline{\mathcal{U}}=\varnothing$ and $q_1>0$ on $\{\rho=-\varepsilon_0\}$.
\end{itemize}
Let $$P_0=\frac{h}{i}X-i(Q_\infty+q_1):\mathcal{D}_h^s:=\{u\in \Hcal_h^s \,|\, P_0u\in \Hcal_h^s\}\to\mathcal{H}_h^s.$$ Dyatlov--Guillarmou \cite[Lemma 3.3]{open} proves that for any $C_1,C_2>0$, $s>s_0$ where $s_0$ depends on $C_1$, $h\in(0,h_0)$ where $h_0$ depends on $C_1,C_2,s$, $P_h(z):=P_0-z$ is a Fredholm operator of index zero for $z\in[-C_2h,C_2h]+i[-C_1h,1]$ and its inverse $R_0(z)=P_h(z)^{-1}$ is a meromorphic family of operators with poles of finite rank. The restriction of $R_0(z)$ to $\Ucal$ is independent of the choice of $Q_\infty$ and $q_1$, and the meromorphic extension of $R(\lambda)$ is given by $h\indic_\Ucal R_0(h\lambda)\indic_\Ucal$. In the proof, another complex absorbing potential $Q_\delta\in\Psi_h^{\comp}$ is introduced to compensate the non-invertibility of $P_0$ near the trapped set with the following properties
\begin{itemize}
\item $Q_\delta=\chi_\delta Q_\delta\chi_\delta$ for some $\chi_\delta\in C^\infty(\mathcal{M})$ supported in a $\delta$-neighborhood of $K$,
\item $\WF_h(Q_\delta)\subset\{|\xi|<\delta\}$;
\item $\{x\in K,\xi=0\}\subset\el_h(Q_\delta)$;
\item $\WF_h(G)\cap\WF_h(Q_\delta)=\varnothing$.
\end{itemize}
Then $\widetilde{P}_h(z):=P_h(z)-iQ_\delta:\mathcal{D}_h^s\to\mathcal{H}^s$ is invertible for $z\in[-C_2h,C_2h]+i[-C_1h,1]$ with inverse $\widetilde{R}_h(z)=\widetilde{P}_h(z)^{-1}:\mathcal{H}_h^s\to\mathcal{H}_h^s$ satisfying the estimate
\begin{equation}
\|\widetilde{R}_h(z)\|_{\mathcal{H}_h^s\to\mathcal{H}_h^s}\leq Ch^{-1}.
\end{equation}
See \cite[Lemma 3.2]{open}. For the proof of both Theorem \ref{t:upper} and \ref{t:lower}, we need various extensions of the resolvent estimates above, which will be described later.

\section{Upper bound on the number of resonances}
\label{s:upper}
In this section, we prove Theorem \ref{t:upper} on the upper bound for the number of resonances. First we make some modification to the setup in \cite{open} described in Section \ref{s:spectral}. The argument in \cite{open} concentrates near the energy surface $\{p=0\}$, but we will be interested near the energy surface $\{p=1\}$ to obtain polynomial upper bound by semiclassical rescaling. In particular, we need to introduce slightly different complex absorbing potential $W$ (corresponding to $Q_\delta$ in Dyatlov--Guillarmou \cite[(3.9)]{open}) which is microlocally supported near $E_0^\ast\cap p^{-1}(1)$ instead of $\{x\in K,\xi=0\}$.

\subsection{Modification of the complex absorbing operator}
First we describe the necessary modification to the complex absorbing potential. We choose $W\in \Psi^{\comp}_h$ satisfying the following properties
\begin{itemize}
\item $W=\chi W\chi$, for some $\chi\in C^\infty(\mathcal{M})$ supported in a neighborhood of $K$ such that $\chi=1$ near $p^{-1}(1)\cap E_0^*$;
\item $\WF_h(W)$ is contained in a neighborhood of $E_0^\ast\cap p^{-1}(1)$;
\item $\WF_h(W)\cap\WF_h(G)=\varnothing$;
\item $E_0^\ast\cap p^{-1}(1)\subset \el_h(W)$;
\item $\rank(W)=\mathcal{O}(h^{-n})$.
\end{itemize}

Now we follow \cite{open} to construct a microlocal weight function and the corresponding anisotropic Sobolev spaces. Let
\begin{align*}
    \widetilde{m}(x,\xi)=(1-\chi_m(x,\xi))m(x,\xi)\log |\xi|
\end{align*}
where $m$ is the weight function defined in Proposition \ref{c:weight}, $\chi_m\in C^\infty(T^*\mathcal{M};[0,1])$ is equal to $1$ near the zero section and supported in some ball $\{|\xi|<C\}$ for some constant $C$.
Define the operator $G=G(h)\in \Psi^{0+}_h(\mathcal{M})$ such that
\begin{itemize}
\item $\sigma_h(G)=\widetilde{m}$;
\item $\WF_h(G)\subset\{\rho>-2\epsilon_0\}$;
\item $\WF_h(G)\cap \WF_h(W)=\varnothing$;
\item $\WF_h(G)\cap\{\rho=-\varepsilon_0\}\cap \{X_1\rho=0\}=\varnothing$.
\end{itemize}
We can then define the anisotropic Sobolev space as $\mathcal{H}^s_h=\exp(-sG(h))L^2(\mathcal{M})$ with norm $\|u\|_{\mathcal{H}^s_h}=\|e^{sG(h)}u\|_{L^2(\mathcal{M})}$.

To make the operator invertible, we take the complex absorbing potentials $Q_\infty \in \Psi^1_h$, $q_1\in C^\infty$
supported outside $\overline{\mathcal{U}}$ as in \cite{open}, such that
\begin{itemize}
    \item $\indic_{\Sigma'}Q_\infty=Q_\infty\indic_{\Sigma'}=0$ for some neighborhood $\Sigma'$ of $\overline{\Sigma}$;
    \item $\{\rho\leq -2\epsilon_0\}\cup (\{\rho=-\epsilon_0\}\cap \{X_1\rho=0\})\subset {\rm ell}_h(Q_\infty)$;
    \item $\WF_h(G)\cap \WF_h(Q_\infty))=\varnothing$;    
    \item $\supp q_1\cap \overline{\mathcal{U}}=\varnothing$ and $q_1>0$ on $\{\rho=-\epsilon_0\}$.
\end{itemize}
Now we define
\begin{align*}
   P_0=\frac{h}{i}X-i(Q_\infty+q_1)\in \Psi^1_h,\quad  \Tilde{P}_h(z)=P_0-iMhW-z\in\Psi^1_h.
\end{align*}
We will be using propagation of singularities and radial estimates to prove $\tilde{P}_h(z)$ is invertible. The potential $Q_\infty$ and $q_1$ are supported outside  $\Ucal$ and will thus not affect the resolvent in $\Ucal$. We need to add them in order to obtain a global estimate on $\Mcal$. The potential $W$ is crucial in this construction in order to get estimates near $E_0^*\cap p^{-1}(1)$.

\subsection{Resolvent estimates}

\begin{thm} 
\label{t:resolvent}
Given $\beta>0$ there exists $M>0$ such that $z\in[1-h,1+h]+i[-\beta h,1]$, $\Tilde{P}_h(z)$ is invertible with $\Tilde{R}_h(z)=(P_0-iMhW-z)^{-1}:\mathcal{H}^s_h\to\mathcal{H}^s_h$ satisfying
$$\|\Tilde{R}_h(z)\|_{\mathcal{H}^s_h\to\mathcal{H}^s_h}\leq \frac{C}{\max(h,{\rm Im}\, z-Ch)}.$$
\end{thm}

We prove Theorem \ref{t:resolvent} following \cite[Lemma 3.2]{open}. The main difference is that we use $MhW$ instead of $Q_\delta$, which requires a further application of propagation estimate near $E_0^\ast\cap \{p=1\}$. The constant $M$ is chosen to be sufficiently large to make the propagation estimate work as in Proposition \ref{p:semiestimate} (2). Though this makes the estimate a little more complicated, it allows us to get slightly better upper bounds for counting of resonances.

\begin{proof}
We apply the propagation estimate Proposition \ref{p:propa} to $P_0-iMhW-z$ with
\begin{align*}
    L=\pi^{-1}(L_-)\cup\pi^{-1}(L_+)\cup \kappa(E_s^*)\cup \kappa(E_u^*)\cup (E_0^*\cap p^{-1}(1)),
\end{align*}
\begin{align*}
    B=0, \quad  \overline{\bigcup\limits_{t\geq 0}e^{-tH_p}(\nbd(\WF_h(A)))}\subset\el_h(B_1).
\end{align*}
Similar to \cite[Lemma 3.2]{open}, by a functional analytic argument we get the estimate
$$\|(P_0-iMhW-z)^{-1}\|_{\mathcal{H}^s_h\to\mathcal{H}^s_h}\leq \frac{C}{h}.$$

On the other hand, if $\Im z\geq C_1h$ for some sufficiently large $C_1>0$, we have
\begin{align*}
    \|(P_0-iMhW-z)^{-1}u\|_{\Hcal^s_h}&\leq \int_0^\infty \|e^{-it(P_0-iMhW-z)}u\|_{\Hcal^s_h}dt\\
    &\leq \int_0^\infty e^{-(\Im z-Ch)t}\|u\|_{\Hcal^s_h}dt\\
    &\leq \frac{1}{\Im z-Ch}\|u\|_{\Hcal^s_h}.
\end{align*}
In the second step we use
\begin{align*}
    \|e^{-it(P_0-iMhW-z)}u\|_{\Hcal^s_h}\leq e^{-(\Im z-Ch)}\|u\|_{\Hcal^s_h}.
\end{align*}
This is because if $u\in \Hcal^s_h$, and $v(t)=e^{-it(P_0-iMhW-z)}u$, then
\begin{align*}
    \partial_t v= -i(P_0-iMhW-z)v,\quad v|_{t=0}=u.
\end{align*}
Thus
\begin{align*}
    &\partial_t\|v(t)\|_{\Hcal^s_h}^2=2\Re \langle e^{sG(h)}\partial_t v,e^{sG(h)}v\rangle\\
    &=2\Im \langle e^{sG(h)}(P_0-iMhW-z) v,e^{sG(h)}v\rangle\\
    &\leq -2\Im z\|v\|_{\Hcal^s_h}^2+2\Im\langle [e^{sG(h)},P_0]v,e^{sG(h)}v\rangle+h|\langle (X^*+X)e^{sG(h)}v,e^{sG(h)}v\rangle|\\
    &\quad\quad\quad +h|\langle e^{sG(h)}MWv,e^{sG(h)}v\rangle|
    -2\Re\langle(Q_\infty+q_1) e^{sG(h)}v,e^{sG(h)}v\rangle\\
    &\leq 2(\tilde{C}h-\Im z)\|v\|_{\Hcal^s_h}^2+2s\langle ( [G,P_0]+\Ocal(h^2)_{\Psi_h^{-1+}}) e^{sG(h)}v,e^{sG(h)}v\rangle\\
    &\leq 2(Ch-\Im z)\|v\|_{\Hcal^s_h}^2,
\end{align*}
and
\begin{equation*}
    \|v(t)\|_{\Hcal^s_h}\leq e^{(Ch-\Im z)t}\|u\|_{\Hcal^s_h}.\qedhere
\end{equation*}
\end{proof}
\subsection{Proof of the upper bound}
Now we prove Theorem \ref{t:upper}, by combining Jensen's inequality and a resolvent identity.

Recall
\begin{align*}
    (P_0-z-iMhW)^{-1}(P_0-z)=I+i(P_0-z-iMhW)^{-1}MhW. 
\end{align*}
Let 
\begin{align*}
    F(z)=I+i(P_0-z-iMhW)^{-1}MhW, 
\end{align*}
then $F(z)$ is a holomorphic family of finite rank operators on the region $[1-C_1'h,1+C_1'h]+i[-C_2'h,1]$.
\begin{itemize}
    \item By the construction of $W$, ${\rm rank}\, W=\mathcal{O} (h^{-n})$ and 
    \begin{align*}
        \|(P_0-z-iMhW)^{-1}\|\lesssim h^{-1},\quad \|W\|\lesssim 1,
    \end{align*}
    we have $|F(z)|\leq \exp(Ch^{-n})$.
    \item For $z_0=1+iC_3h$ for sufficiently large $C_3>0$, we have 
    \begin{align*}
         \|(P_0-z_0-iMhW)^{-1}MhW\|\leq \frac{1}{2},
    \end{align*}
    and $|F(z_0)| \geq \exp(-Ch^{-n})$.
\end{itemize}
We rescale $[1-C_1h,1+C_1h]+i[-C_2h, 2C_3h]$ to $[h^{-1}-C_1,h^{-1}+C_1]+i[-C_2, 2C_3]$ and map it to the unit disc by a conformal map, with $z_0$ mapping to $0$, and a slightly larger set $ \subset [1-C_1'h,1+C_1'h]+i[-C_2'h, 2C_3'h]$ that is mapped to the disc $\{|z|\leq 2\}$.
By Jensen's formula, for any holomorphic function $f(z)$ on an open set containing the disc $\{|z|\leq r\}$,
\begin{align*}
    \log|f(0)|+\int_0^r \frac{n(r)}{r}dr=\frac{1}{2\pi}\int_0^{2\pi} \log |f(re^{i\theta})|d\theta,
\end{align*}
where $n(r)$ is the number of zeros inside the disc $\{|z|<r\}$. We apply this to the $F(z)$ on the rescaled region to get
\begin{align*}
    n(1)\leq\frac{1}{\log 2}\int_1^2\frac{n(r)}{r}dr=\mathcal{O}(h^{-n}).
\end{align*}
That is, the number of resonances has an upper bound
\begin{align*}
    \#{\rm Res}(P)\cap [1-C_1h,1+C_1h]+i[-C_2h,\infty)=\mathcal{O}(h^{-n}).
\end{align*}
Taking $E=h^{-1}$ proves our theorem \ref{t:upper}.

\section{Lower bound on the number of resonances}
\label{s:lower}
In this section we prove a local trace formula generalizing \cite{local} and deduce the lower bound on the number of resonances, Theorem \ref{t:lower}. Here we go back to the setup from \cite{open}, working near the energy level $\{p=0\}$, see Section \ref{s:spectral}. But as in Section \ref{s:upper}, we use $W$ instead of $Q_\delta$ to denote the complex absorbing potential. In particular, we require that 
\begin{itemize}
    \item $W=\chi W\chi$ for some $\chi \in C^\infty(\Ucal)$ supported in a neighborhood of $K$;
    \item $\{x\in K,\xi=0\}\subset \el_h(W)$;
    \item $\WF_h(W)\subset \{|\xi|<\delta\}$ for some $0<\delta<1$;
    \item $\rank (W)=\Ocal(h^{-n})$.
\end{itemize}
Let $P_h(z)=\frac{h}{i}X-i(Q_\infty+q_1)-z\in \Psi^1_h$ and $\tilde{P}_h(z)=P_h(z)-iW\in \Psi^1_h$. Then we have (see \cite[Lemma 3.2]{open}), but with a slightly larger window for the real part of $z$ as in \cite[Proposition 3.4]{zeta}:

\begin{lem}
Fix any $\varepsilon\in(0,1)$ and $C_1>0$, there exists $s>0$, $h_0>0$ such that for any $h\in(0,h_0)$ and $z\in [-h^\varepsilon,h^\varepsilon]+i[-C_1h,1]$, $\tilde{P}_h(z):\Dcal^s_h\to \Hcal^s_h$ is invertible, with a bounded inverse
\begin{equation}
\label{e:modified-resolvent}
    \tilde{R}_h(z):=\tilde{P}_h(z)^{-1}:\Hcal^s_h\to \Hcal^s_h
\end{equation}
and
\begin{equation}
\label{e:modified-resolvent-est}
    \|\tilde{R}_h(z)\|_{\Hcal^s_h\to \Hcal^s_h}\leq Ch^{-1}.
\end{equation}
\end{lem}

\begin{rem}
Here we notice that the real part of principal symbol of $\tilde{P}_h(z)$, $p-z$ may depend on $z\in[-h^\varepsilon,h^\varepsilon]$ and thus on $h$, but the corresponding Hamiltonian flow does not depend on $h$ and the extra term $-z$ is eliminated by any commutator. All the propagation estimates still work as in \cite{open} and thus we can extend \cite[Lemma 3.2]{open} to $z\in[-h^\varepsilon,h^\varepsilon]$.
\end{rem}

\subsection{Support conditions}
We present several lemmas concerning the support of a function when certain propagator or resolvent is applied. 

\begin{lem}
\label{l:support1}
Let $f\in C_c^\infty(\Sigma_+)$, then
\begin{itemize}
\item[(a)] for $t\geq 0$, 
\begin{equation}
\label{e:suppw}
    \supp e^{-ith^{-1}(-ihX-i(q_1+W))}f\subset \varphi^t(\supp f)\cup \left(\bigcup\limits_{s=0}^t\varphi^s(\supp W)\right)\subset \Sigma_+,
\end{equation}
where $\supp W$ is defined as the complement of the largest open set $U$ such that $\chi W = W\chi=0$ for any $\chi\in C_c^\infty(U)$.
\item[(b)] for $t\geq0$,
\begin{equation}
e^{-ith^{-1}\tilde{P}_h(0)}f=e^{-ith^{-1}(-ihX-i(q_1+W))}f
\end{equation}
\item[(c)] for $\Im z\geq C_0h$ with some large constant $C_0>0$, we have
\begin{equation}
\label{e:resolvent-noqinfty}
\tilde{R}_h(z)f=\frac{i}{h}\int_0^\infty e^{-ith^{-1}\tilde{P}_h(z)}fdt
=\frac{i}{h}\int_0^\infty e^{-ith^{-1}(-ihX-z-i(q_1+W))}fdt
\end{equation}
and thus
\begin{equation}
\label{e:resolvent-supp}
\supp\tilde{R}_h(z)f\subset\overline{\Sigma}_+.
\end{equation}
\end{itemize}
\end{lem}

\begin{proof} For part (a), let 
$$w=w(t)=e^{-ith^{-1}(-ihX-i(q_1+W))}f\in L^2(\Mcal),\quad\quad  t\geq0$$
be the unique solution of 
\begin{align*}
    \partial_tw=-ih^{-1}(-ihX-i(q_1+W))w,\quad w|_{t=0}=f,
\end{align*}
and let
\begin{align*}
    v(t,x)=e^{h^{-1}\int_0^t q_1(\varphi^s(x))ds}w(t,\varphi^t(x)),
\end{align*}
then
\begin{align*}
    \partial_t v=-h^{-1} e^{h^{-1}\int_0^t q_1(\varphi^s(x))ds}Ww(t,\varphi^t(x)).
\end{align*}
Since $\supp Ww\subset \supp W$ and $v(0,x)=f(x)$, we have
\begin{align*}
    \supp v(t,\cdot)\subset (\supp f)\cup\left( \bigcup\limits_{s=0}^t \varphi^{-s}(\supp W)\right).
\end{align*}
Thus for $t\geq0$,
\begin{align*}
    \supp w(t,\cdot)\subset \varphi^t(\supp v(t,\cdot))\subset\varphi^t(\supp f)\cup\bigcup\limits_{s=0}^t \varphi^{s}(\supp W)\subset\Sigma_+.
\end{align*}
This proves part (a). Part (b) follows from part (a) by the choice of $Q_\infty$ that $Q_\infty\indic_{\Sigma_+}=0$, which implies that for $t\geq0$, $Q_\infty w(t)=0$ and thus $w$ also satisfies
\begin{align*}
    \partial_tw=-ih^{-1}\tilde{P}_h(0)w=-ih^{-1}(-ihX-i(q_1+W+Q_\infty))w,\quad w|_{t=0}=f.
\end{align*}
Part (c) is a direct consequence of part (b).
\end{proof}

\begin{lem}
\label{l:support2}
Let $f\in C_c^\infty(\Sigma_+)$.
\begin{itemize}
\item[(a)]  If $\supp f\cap\overline{\Ucal}=\varnothing$, then
for $t\geq 0$, 
$$e^{-ith^{-1}(-ihX-iW)}f=e^{-tX}f=f\circ\varphi^{-t}$$
and thus
\begin{equation}
    \supp e^{-ith^{-1}(-ihX-iW)}f\cap\overline{\mathcal{U}}=\varnothing;
\end{equation}
\item[(b)] for $t\geq0$,
$$e^{-ith^{-1}(-ihX-i(q_1+W))}f|_\mathcal{U}=
e^{-ith^{-1}(-ihX-iW)}f|_\mathcal{U}.$$
\item[(c)] for $\Im z\geq C_0h$ with some large constant $C_0>0$, we have
\begin{equation}
\label{e:resolvent-onlyw}
(\tilde{R}_h(z)f)|_\Ucal=\frac{i}{h}\int_0^\infty \left.\left(e^{-ith^{-1}\tilde{P}_h(z)}f\right)\right|_\Ucal dt
=\frac{i}{h}\int_0^\infty\left.
\left(e^{-ith^{-1}(-ihX-z-iW)}f\right)\right|_\Ucal dt.
\end{equation}
\end{itemize}

\end{lem}
\begin{proof}
For part (a), let $w=e^{-tX}f$, then for $t\geq0$, by convexity of $\overline{\Ucal}$,
$$\supp w=\varphi^t(\supp f)\subset\Sigma_+\setminus\overline{\Ucal}.$$
Therefore $w$ also satisfies
\begin{align*}
    \partial_tw=-ih^{-1}(-ihX-iW)w,\quad w|_{t=0}=f.
\end{align*}
For part (b), we use the following identity
\begin{align*}
     &\;e^{-ith^{-1}(-ihX-i(q_1+W))}- e^{-ith^{-1}(-ihX-iW)}\\
    =&\;-h^{-1}\int_0^t e^{-i(t-s)h^{-1}(-ihX-iW)}q_1 e^{-ish^{-1}(-ihX-i(q_1+W))}ds.
\end{align*}
and by \eqref{e:suppw},
$$\supp q_1e^{-ish^{-1}(-ihX-i(q_1+W))}f\subset \supp q_1\cap\Sigma_+.$$
Since $\supp q_1\cap\overline{\Ucal}=\varnothing$, we can apply part (a) to $q_1e^{-ish^{-1}(-ihX-i(q_1+W))}f$ to get
$$\supp(e^{-ith^{-1}(-ihX-i(q_1+W))}f- e^{-ith^{-1}(-ihX-iW)}f)\cap\overline{\Ucal}=\varnothing.$$
Part (c) is a direct corollary of part (b) and \eqref{e:resolvent-noqinfty}.
\end{proof}

\subsection{A wavefront set estimate for the modified resolvent}
Here we present the estimate on the semiclassical wavefront set for $\tilde{R}_h(z)$ on fiber infinity following \cite{flattrace}.

\begin{lem}
Fix any $\varepsilon\in(0,1)$ and $C_1>0$, for $z\in[-h^\varepsilon,h^\varepsilon]+i[-C_1h,1]$, the modified resolvent \eqref{e:modified-resolvent} satisfies
\begin{equation}
\label{e:mr-wf}
    {\rm WF}'_h(\tilde{R}_h(z))\cap S^*(\mathcal{U}\times\mathcal{U})\subset \kappa(\Delta(T^*\mathcal{U})\cup \Omega_+\cup (E^*_+\times E^*_-)\setminus\{0\})
\end{equation}
where $\Omega_+$ is the flowout
\begin{align*}
    \Omega_+=\{(e^{tH_p}(y,\eta),y,\eta)\,:\, p(y,\eta)=0, y\in\mathcal{U},\varphi^t(y)\in\mathcal{U}, t\geq0\}.
\end{align*}
\end{lem}
\begin{proof} The general strategy is the same as \cite[Proposition 2.1]{flattrace} and we only emphasize the difference and refer the reader to \cite{flattrace} for more details. Let $p^{-1}(0)=\{(x,\xi)\in T^*\Mcal: p(x,\xi)=0\}$. We only need to prove a weaker statement
\begin{equation}
\label{e:mr-wf-weak}
\WF_h'(\tilde{R}_h(z))\cap S^*(\mathcal{U}\times\mathcal{U})\subset \kappa(\Delta(T^*\mathcal{U})\cup \Omega_+\cup (E^*_+\times p^{-1}(0))\setminus\{0\})
\end{equation}
and then by a duality argument on the flow generated by $-X$, we can obtain 
$$\WF_h'(\tilde{R}_h(z))\cap S^*(\mathcal{U}\times\mathcal{U})\subset \kappa(\Delta(T^*\mathcal{U})\cup \Omega_+\cup (p^{-1}(0)\times E_-^\ast)\setminus\{0\})$$
and thus finishes the proof of \eqref{e:mr-wf}. Now we prove \eqref{e:mr-wf-weak}, that is, given
$$(x_0,\xi_0,y_0,\eta_0)\in \{(x,\xi,y,\eta)\in T^*\Ucal\times T^*\Ucal:|(\xi,\eta)|=1\}\setminus (\Delta\cup \Omega_+\cup (E_+^*\times p^{-1}(0))),$$
there exists $\chi,\psi\in C_c^\infty(\mathcal{U})$, such that $\chi(x_0)\neq0$, $\psi(y_0)\neq0$ and for $(\xi,\eta)$ in a conic neighbourhood of $(\xi_0,\eta_0)$, we have uniformly
\begin{equation}
\label{e:wfestimate} 
\int\chi(x)e^{-ix\cdot \xi/h}\Tilde{R}_h(z)(\psi(x)e^{ix\cdot \eta/h})dx=\mathcal{O}(h^\infty)\langle\xi,\eta\rangle^{-\infty}, \quad\langle \xi,\eta\rangle\to\infty.
\end{equation}
By propagation estimates, there exist $r>0$, and neighbourhoods $U$ of $(x_0,r\xi_0)$ and $V$ of $(y_0,r\eta_0)$, and $A,B\in\Psi^0_{h}(\mathcal{M})$ such that
\begin{align*}
    \|Au\|_{\mathcal{H}^s_{h}}\leq Ch^{-1}\|B(-ihX-z-i(Q_{\infty}+q_{1}+W ))u\|_{\mathcal{H}^s_{h}}+\mathcal{O}(h^\infty)\|u\|_{H_{h}^{-N}},\\
    U\subset\el_{h}(A),\quad (\{|\xi|\leq 1\}\cup V)\cap \WF_{h}(B)=\varnothing.
\end{align*}
Here one should think of $A$ microlocally supported near $(x_0,r\xi_0)$ and $B$ microlocally supported in a neighbourhood of $\{e^{-tH_p}(x_0,r\xi_0): t\geq 0\}$. The condition that $(x_0,\xi_0,y_0,\eta_0)\notin E_+^*\times p^{-1}(0)$ guarantees $\WF_h(B)\cap \{\xi\leq 1\}=\varnothing$ for sufficiently large $r$.

As in \cite{flattrace}, we introduce another independent semiclassical parameter $\tilde{h}\in(0,1)$ and $\tilde{h}\to0$ plays the role of $|(\xi,\eta)|\to+\infty$. Replacing $h$ by $h\tilde{h}$, we get $A_{\tilde{h}},B_{\tilde{h}}\in\Psi^0_{h\Tilde{h}}(\Mcal)$ such that
\begin{align*}
    \|A_{\tilde{h}}u\|_{\mathcal{H}^s_{h\tilde{h}}}\leq C(h\tilde{h})^{-1}\|B_{\tilde{h}}(-ih\tilde{h}X-\tilde{h}z-i(Q_{\infty,\tilde{h}}+q_{1}+W_{\tilde{h}} )u\|_{\mathcal{H}^s_{h\tilde{h}}}+\mathcal{O}((h\tilde{h})^\infty)\|u\|_{H_{h\tilde{h}}^{-N}},\\
    U\subset\el_{h\Tilde{h}}(A_{\tilde{h}}),\quad (\{|\xi|\leq 1\}\cup V)\cap \WF_{h\Tilde{h}}(B_{\tilde{h}})=\varnothing.
\end{align*}

Without loss of generality we may assume 
\begin{align*}
   A_{\tilde{h}}=\Op_{h\tilde{h}}(a),\quad B_{\tilde{h}}=\Op_{h\tilde{h}}(b),\quad W=\Op_h(w)
\end{align*}
with symbols $b\in S^0$ and $a,q\in C_c^\infty$ independent of $h$, and $\{|\xi|\leq 1\}\cap\supp b=\varnothing$.

The $W$ term can be dealed by \cite[Lemma 2.3]{flattrace}, and we obtain
\begin{align*}
    \|B_{\tilde{h}}Wu\|_{H^N_{h\tilde{h}}}=\mathcal{O}(h^\infty \tilde{h}^\infty)\|u\|_{H^{-N}_{h\tilde{h}}}.
\end{align*}

Taking $u(x)=(-ihX-z-i(\tilde{h}^{-1}Q_{\infty,\tilde{h}}+\tilde{h}^{-1}q_{1}+W))^{-1}(\psi(x)e^{ix\cdot r\eta_0/h\Tilde{h}})$ where $\supp\psi\times \{\eta_0\}\subset V$, by wavefront conditions
\begin{align*}
   B_{\tilde{h}}W_{\tilde{h}}=\Ocal(h^\infty\tilde{h}^\infty),\quad B_{\tilde{h}}(\psi(x)e^{ix\cdot r\eta_0/h\Tilde{h}})=\Ocal(h^\infty\tilde{h}^\infty).
\end{align*}

Thus
\begin{align*}
    &\|A_{\tilde{h}}u\|_{\mathcal{H}^s_{h\tilde{h}}}\\
    &\leq Ch^{-1}\|B_{\tilde{h}}(-ihX-z-i(\tilde{h}^{-1}Q_{\infty,\tilde{h}}+\tilde{h}^{-1}q_{1}+W))u\|_{\Hcal^s_{h\tilde{h}}}+\mathcal{O}(h^\infty\tilde{h}^\infty)\|u\|_{H^{-N}_{h\tilde{h}}}\\
    &\leq Ch^{-1}\|B_{\tilde{h}}(\psi(x)e^{ix\cdot r\eta_0/h\Tilde{h}})\|_{\Hcal^s_{h\tilde{h}}}+\mathcal{O}(h^\infty\tilde{h}^\infty)\|u\|_{H^{-N}_{h\tilde{h}}}
    \\
    &=\mathcal{O}(h^\infty\Tilde{h}^\infty).
 \end{align*}
This means $\WF_{h\Tilde{h}}(u)\cap U=\varnothing$, thus
\begin{align*}
        \int\chi(x)e^{-ix\cdot \xi_0/h\Tilde{h}}(-ihX-z-i(\tilde{h}^{-1}Q_{\infty,\tilde{h}}+\tilde{h}^{-1}q_{1}+W))^{-1}(\psi(x)e^{ix\cdot r\eta_0/h\Tilde{h}})dx=\mathcal{O}(h^\infty\Tilde{h}^\infty).
\end{align*}
Now we observe that \eqref{e:resolvent-onlyw} gives that $\mathbbm{1}_\Ucal\tilde{R}_h(z)\mathbbm{1}_\Ucal$ and $\mathbbm{1}_\Ucal(-ihX-z-iW)^{-1}\mathbbm{1}_\Ucal$ agree for $\Im z\geq C_0h$. A similar argument also applies to $\mathbbm{1}_\Ucal(-ihX-z-i(\tilde{h}^{-1}Q_{\infty,\tilde{h}}+\tilde{h}^{-1}q_{1}+W))^{-1}\mathbbm{1}_\Ucal$. Therefore $\mathbbm{1}_\Ucal\tilde{R}_h(z)\mathbbm{1}_\Ucal$  and $\mathbbm{1}_\Ucal(-ihX-z-i(\tilde{h}^{-1}Q_{\infty,\tilde{h}}+\tilde{h}^{-1}q_{1}+W))^{-1}\mathbbm{1}_\Ucal$ are both analytic continuations of 
\begin{align*}
     \mathbbm{1}_\Ucal(-ihX-z-iW)^{-1}\mathbbm{1}_\Ucal: C_c^\infty(\Ucal)\to\mathcal{D}'(\Ucal)
\end{align*}
in $z\in [-h^\varepsilon,h^\varepsilon]+i[-C_1h,1]$ and thus they are equal. Therefore
\begin{equation}
\label{e:hinfty}
\int\chi(x)e^{-ix\cdot \xi_0/h\Tilde{h}}\tilde{R}_h(z)(\psi(x)e^{ix\cdot r\eta_0/h\Tilde{h}})dx=\mathcal{O}(h^\infty\Tilde{h}^\infty).
\end{equation}
Moreover, \eqref{e:hinfty} holds uniformly with $(\xi_0,\eta_0)$ replaced by $(\xi,\eta)$ in a neighbourhood of $(\xi_0,\eta_0)$. This finishes the proof of \eqref{e:wfestimate} and thus the lemma.
\end{proof}

\subsection{Flat trace estimates}
\label{s:flattrace}
In this section we prove the following flat trace estimate following \cite[Section 3]{local} and \cite{flattrace}.
\begin{prop}
\label{trbound1}
Let $\chi\in C_c^\infty(\Ucal)$ such that $\chi=1$ near $K$, $t_0\in(0,\inf T_\gamma)$ be sufficiently small such that for $t\in[-t_0,t_0]$, $\varphi^t(\supp\chi)\subset\Ucal$ , then
\begin{align*}
    T(z)={\rm tr}^\flat (\chi e^{-it_0h^{-1}\tilde{P}_h(z)} \tilde{R}_h(z)\chi)
\end{align*}
is well-defined and holomorphic in $z$ in $[-h^\varepsilon,h^\varepsilon]+i[-C_1h,1]$. Moreover, we have 
\begin{align*}
    T(z)=\mathcal{O}(h^{-2n-2}).
\end{align*}
\end{prop}
We use  Lemma \ref{l:semiwf-flattrace} to prove Proposition \ref{trbound1}. First we check the semiclassical wavefront set condition for $e^{-it_0h^{-1}\tilde{P}_h(z)}\tilde{R}_h(z)$:
\begin{lem} 
Under the hypothesis of Proposition \ref{trbound1}, we have
\begin{align*}
{\rm WF}_h'(\chi e^{-it_0h^{-1}\tilde{P}_h(z)}\tilde{R}_h(z)\chi)\cap S^*(\mathcal{M}\times\mathcal{M})&\subset  \\
\kappa\big(\{(x,\xi,y,\eta)\,:\,(e^{-t_0H_p}(x,\xi),y,\eta)\in \Delta(T^*\mathcal{U})\cup&\; \Omega_+\cup E^*_+\times E^*_-\setminus\{0\} \text{ or } \xi=0,\eta\neq 0\}\big).
\end{align*}
In particular, $\WF_h'(\chi e^{-it_0h^{-1}\tilde{P}_h(z)}\tilde{R}_h(z)\chi)\cap\Delta(S^\ast\mathcal{M})=\varnothing$.
\end{lem}
\begin{proof}
By \eqref{e:mr-wf}, we have
\begin{align*}
     {\rm WF}_h'(e^{-t_0X}\tilde{R}_h(z))\cap S^*(\mathcal{U}\times\mathcal{U})&\subset \\
\kappa\big(\{(x,\xi,y,\eta)\,:\,&(e^{-t_0H_p}(x,\xi),y,\eta)\in \Delta(T^*\mathcal{U})\cup \Omega_+\cup E^*_+\times E^*_- \}\setminus\{0\}\big).
\end{align*}
Here the propagators $e^{-t_0X}$ and $e^{-it_0h^{-1}\tilde{P}_h(0)}$ are related by
\begin{equation}
\label{e:propagator-relation}
    e^{-t_0X}- e^{-it_0h^{-1}\tilde{P}_h(0)}=h^{-1}\int_0^{t_0}e^{-(t_0-t)X}(Q_\infty+q_1+W)e^{-ith^{-1}\tilde{P}_h(0)}dt.
\end{equation}
The term with $W$ can be estimated by
\begin{align*}
    \WF_h'(e^{-(t_0-t)X}We^{-ith^{-1}\tilde{P}_h(0)}\tilde{R}_h(z))\cap S^*(\Ucal\times\Ucal)\subset \kappa\big(\{(x,0,y,\eta):x,y\in\mathcal{U},\eta\neq0\}\big)
\end{align*}
since $W$ is microlocally compactly supported.

Moreover, by the assumptions on $t_0$, for any $t\in[0,t_0]$,
$$\chi e^{-(t_0-t)X}(1-\indic_{\Ucal})=0.$$
Now since $\indic_{\Ucal}(Q_\infty+q_1)=0$, we have the corresponding term vanishes:
\begin{equation}
\label{e:propagation-vanish}
\chi e^{-(t_0-t)X}(Q_\infty+q_1)e^{-ith^{-1}\tilde{P}_h}\tilde{R}_h(z)\chi=0.
\end{equation}
This concludes the proof.
\end{proof}

Now to finish the proof of Proposition \ref{trbound1}, we use  \eqref{e:modified-resolvent-est} to deduce the $L^2$ bound: For any $A,B\in\Psi^{\comp}_h(\mathcal{M})$, we need to show
\begin{equation}
\label{e:mr-l2}
\|A\chi e^{-it_0h^{-1}\tilde{P}_h(z)} \tilde{R}_h(z)\chi B\|_{L^2\to L^2}=\mathcal{O}(h^{-2}).
\end{equation}
Again, by an analogue of \eqref{e:propagator-relation} 
\begin{align*}
A\chi e^{-it_0h^{-1}\tilde{P}_h(0)} \tilde{R}_h(z)\chi B=
&\;A\chi e^{-t_0X}\tilde{R}_h(z)\chi B\\
&\;\;-h^{-1}\int_0^{t_0}A\chi e^{-ith^{-1}\tilde{P}_h(0)}(W+Q_\infty+q_1) e^{-(t_0-t)X}\tilde{R}_h(z)\chi Bdt.
\end{align*}
Here the terms with $Q_\infty$ and $q_1$ vanish:
\begin{itemize}
\item By \eqref{e:resolvent-supp} and $Q_\infty\indic_{\overline{\Sigma}_+}=0$, we see that
$$Q_\infty e^{-(t_0-t)X}\tilde{R}_h(z)\chi=0.$$
\item We claim
\begin{equation}\label{e:vanish_q1}
    \chi e^{-ith^{-1}\tilde{P}_h(0)}q_1 e^{-(t_0-t)X}\tilde{R}_h(z)\chi 
=\chi e^{-tX}q_1e^{-(t_0-t)X}\tilde{R}_h(z)\chi=0.
\end{equation}
Since it is holomorphic in $z$, it suffices to show \eqref{e:vanish_q1} for $\Im z>C_0h$ for some sufficiently large $C_0>0$. In this case, one can write
\begin{align*}
    \chi e^{-ith^{-1}\tilde{P}_h(0)}q_1 e^{-(t_0-t)X}\tilde{R}_h(z)\chi=ih^{-1}\int_0^\infty \chi e^{-ith^{-1}\tilde{P}_h(0)}q_1 e^{-(t_0-t)X}e^{-ish^{-1}\tilde{P}_h(z)}\chi ds.
\end{align*}
Thus it suffices to show for any $f\in C_c^\infty(\Mcal)$,
\begin{equation}\label{e:vanish_q1_2}
    \chi e^{-ith^{-1}\tilde{P}_h(0)}q_1 e^{-(t_0-t)X}e^{-ish^{-1}\tilde{P}_h(z)}\chi f=\chi e^{-tX}q_1 e^{-(t_0-t)X}e^{-ish^{-1}\tilde{P}_h(z)}\chi f=0.
\end{equation}
By Lemma \ref{l:support1} (a)(b), $\supp  e^{-(t_0-t)X}e^{-ish^{-1}\tilde{P}_h(z)}\chi f \subset \Sigma_+$, thus \eqref{e:vanish_q1_2} follows from Lemma \ref{l:support1}\;(b) and Lemma \ref{l:support2}\;(a)(b).
\end{itemize}
For the other terms, by \eqref{e:modified-resolvent-est}, we conclude that for any $\tilde{A},\tilde{B}\in\Psi^{\comp}_h(\mathcal{M})$,
\begin{equation}
\label{e:mr-cutoff-est}
\|\tilde{A}\tilde{R}_h(z)\tilde{B}\|_{L^2\to L^2}=\mathcal{O}(h^{-1}),
\end{equation} 
since on compact subset of phase space, the $\mathcal{H}_h^s$ norm is equivalent to the $L^2$ norm. Applying \eqref{e:mr-cutoff-est} to $\tilde{A}=e^{t_0X}A\chi e^{-t_0X}$ and $\tilde{B}=\chi B$, we get the estimate for the first term:
$$\|A\chi e^{-t_0X}\tilde{R}_h(z)\chi B\|_{L^2\to L^2}=\mathcal{O}(h^{-1}).$$
Applying \eqref{e:mr-cutoff-est} to $\tilde{A}=e^{(t_0-t)X}We^{-(t_0-t)X}$ and $\tilde{B}=\chi B$ we get the estimate 
$$\|A\chi e^{-ith^{-1}\tilde{P}_h(0)}We^{-(t_0-t)X}\tilde{R}_h(z)\chi B\|_{L^2\to L^2}=\mathcal{O}(h^{-1}).$$
Combining the above, we get \eqref{e:mr-l2} and thus finish the proof of Proposition \ref{trbound1}.

\subsection{Proof of the local trace formula}
Now we give the proof of Theorem \ref{t:localtrace} which is essentially the same as \cite[Section 4]{local}. Here we only emphasize the difference and refer to \cite[Section 4]{local} for more details. First as in \cite[Section 4]{local}, we decompose
\begin{align*}
    \Omega=\{\lambda\in\CC:-A\leq \Im \lambda\leq B\}
\end{align*}
into $\Omega=\bigcup\limits_{k\in\ZZ}\Omega_k$ where
\begin{align*}
    \Omega_0=\Omega\cap\{-1\leq\Re\lambda\leq 1 \},\quad \Omega_{\pm k}=\Omega\cap\{k\leq \pm\Re\lambda\leq k+1\},k>0.
\end{align*}
As in \cite[Section 4.2]{local}, let $\gamma_k=\partial\Omega_k$ be the contours with counterclockwise orientation; $\gamma^i_k,i=1,2,3,4$ be different parts of the contour; and $\tilde{\gamma}_k^i$ be modifications of $\gamma_k^i$ in the neighbourhood of distance $1$ to avoid the resonances (to be chosen later).

Let $\chi\in C_c^\infty(\Ucal)$ be a cutoff function which equals to 1 near the trapped set $K$. We recall the Atiyah-Bott-Guillemin trace formula (\cite[(4.6)]{open}), as a distribution on $(0,\infty)$,
\begin{align*}
    \Tr^\flat \chi e^{-itP}\chi=\sum\limits_{\gamma}\frac{T_\gamma^\#\delta(t-T_\gamma)}{|\det(I-\mathcal{P}_\gamma)|},\quad t>0,
\end{align*}
If we define
\begin{align*}
    \zeta_1(\lambda)=\exp\left(-\sum\limits_\gamma\frac{T_\gamma^\# e^{i\lambda T_\gamma}}{T_\gamma|\det(I-\mathcal{P}_\gamma)|}\right), \quad\Im\lambda\gg1
\end{align*}
then for $t_0>0$ small enough as in Proposition \ref{trbound1} we have $\supp \Tr^\flat(\chi e^{-itP}\chi) \subset(t_0,\infty)$ and thus for $\Im\lambda\gg 1$,
$$\frac{d}{d\lambda}\log\zeta_1(\lambda)=\frac{1}{i}\sum\limits_{\gamma}\frac{T_\gamma^\# e^{i\lambda T_\gamma}}{|\det(I-\mathcal{P}_\gamma)|}=\frac{1}{i}\int_0^\infty e^{i\lambda t} \Tr^\flat(\chi e^{-itP}\chi) dt=\frac{1}{i}\Tr^\flat \int_{t_0}^\infty\chi e^{-it(P-\lambda)}\chi dt$$
which means
\begin{equation}
\label{e:dlogzeta1}
\frac{d}{d\lambda}\log\zeta_1(\lambda)=-\Tr^\flat (\chi e^{-it_0(P-\lambda)}(P-\lambda)^{-1}\chi).
\end{equation}
This allows us to meromorphically continue $\zeta_1(\lambda)$ to the whole complex plane $\mathbb{C}$ since $e^{-it_0(P-\lambda)}(P-\lambda)^{-1}$ satisfies a suitable wavefront set condition (see \cite[\S 4.1]{open}).

Now the following estimate allows us to choose suitable contours $\tilde{\gamma}_k^i$ to avoid resonances:
\begin{lem}
Fix $0<\eta<1,\epsilon>0$, then we can choose the contours $\tilde{\gamma}_k^i$ in an $\eta$-neighborhood of  $\gamma_k^i$ such that 
\begin{align*}
    \left|\frac{d}{d\lambda}\log\zeta_1(\lambda)\right|=\Ocal(\langle \lambda\rangle^{2n+1+\epsilon}),\quad \lambda\in\tilde{\gamma}_k^i.
\end{align*}
\end{lem}
\begin{proof}
Let $z\in [-h^\epsilon,h^\epsilon]+i[-C_1h,1]$, $\zeta_h(z)=\zeta_1(z/h)$, $\lambda=z/h$, it then suffices to prove that we can choose contours $\tilde{\gamma}_k^i$ such that when $z\in h\tilde{\gamma}_k^i$,
\begin{align}
\label{dzetah}
    \left|\frac{d}{dz}\log\zeta_h(z)\right|=\Ocal(h^{-2n-2}).
\end{align}
Here by \eqref{e:dlogzeta1},
$$\frac{d}{dz}\log\zeta_h(z)
=-\Tr^\flat(\chi e^{-it_0h^{-1}P_h(z)}R_h(z)\chi)$$
where we can decompose
\begin{align*}
    \chi e^{-it_0h^{-1}P_h(z)}R_h(z)\chi=& \chi e^{-it_0h^{-1}\tilde{P}_h(z)}\tilde{R}_h(z)\chi+\chi(R_h(z)-\tilde{R}_h(z))\chi\\
    &-\frac{i}{h}\int_0^{t_0}\chi(e^{-ish^{-1}P_h(z)}-e^{-ish^{-1}\tilde{P}_h(z)})\chi ds.
\end{align*}
For the first term, we use Proposition \ref{trbound1} to get
\begin{align*}
    \Tr^\flat(\chi e^{-it_0h^{-1}\tilde{P}_h(z)}\tilde{R}_h(z)\chi)=\Ocal(h^{-2n-2}).
\end{align*}
The integrand in the last term is a finite rank operator:
$$\chi(e^{-ith^{-1}P_h(z)}-e^{-ith^{-1}\tilde{P}_h(z)})\chi
=h^{-1}\int_0^t \chi e^{-i(t-s)h^{-1}P_h(z)}We^{-ish^{-1}\tilde{P}_h(z)} \chi   ds$$
where by assumption $\rank(W)=\mathcal{O}(h^{-n})$, so 
\begin{align*}
    &\;\Tr^\flat\chi(e^{-ith^{-1}P_h(z)}-e^{-ith^{-1}\tilde{P}_h(z)})\chi\\
    =&\;h^{-1}\Tr\int_0^t \chi e^{-i(t-s)h^{-1}P_h(z)}We^{-ish^{-1}\tilde{P}_h(z)} \chi   ds=\Ocal(h^{-n-1}).
\end{align*}
It remains to estimate $\Tr^\flat\chi(R_h(z)-\tilde{R}_h(z))\chi$ away from the resonances. Let 
$$R_h^0(z)=(-ihX-i(Q_\infty+q_1)-z)^{-1}$$
by \cite[Lemma 3.4]{open}
$$\chi R_h(z)\chi =\chi R_h^0(z)\chi.$$
Note $R_h^0(z)-\tilde{R}_h(z)=-i R_h^0(z)W\tilde{R}_h(z)$ is a finite rank operator. We claim for appropriate $\chi\in C_c^\infty(\Ucal)$,
\begin{equation}\label{e:trace_nochi}
    \Tr(\chi(R_h^0(z)-\tilde{R}_h(z))\chi)=\Tr(R_h^0(z)-\tilde{R}_h(z)).
\end{equation}
Since both sides are meromorphic in $z$, it suffices to check for $\Im z>C_0 h$ for some sufficiently large $C_0>0$. We claim the integration kernel of $R_h^0(z)-\tilde{R}_h(z)$ is supported in
\begin{align}\label{e:supp_difference}
    \overline{\bigcup\limits_{t\geq 0}\varphi^t(\supp W)\times \bigcup\limits_{t\leq 0}\varphi^t(\supp W)}.
\end{align}
We write 
\begin{align*}
    &R_h^0(z)-\tilde{R}_h(z)=-i \tilde{R}_h(z)WR_h^0(z)\\
    &=ih^2\left(\int_0^\infty e^{-ish^{-1}(-ihX -i(Q_\infty+q_1+W)-z)}ds\right) W \left(\int_0^\infty e^{-ith^{-1}(-ihX -i(Q_\infty+q_1)-z)}dt\right).
\end{align*}
It suffices to show the integration kernel of $ e^{-ish^{-1}(-ihX -i(Q_\infty+q_1+W))} W e^{-ith^{-1}(-ihX -i(Q_\infty+q_1))} $ is supported in \eqref{e:supp_difference} for $s,t\geq 0$, which follows from the following facts: for any $g\in C^\infty(\Mcal)$,
\begin{itemize}
    \item 
    By Lemma \ref{l:support1} (b), $e^{-ish^{-1}(-ihX -i(Q_\infty+q_1+W))} W g=e^{-ish^{-1}(-ihX -i(q_1+W))} Wg$. By Lemma \ref{l:support1} (a), 
    $$\supp e^{-ish^{-1}(-ihX -i(q_1+W))} Wg \subset \bigcup\limits_{s'=0}^s \varphi^{s'}(\supp W).$$
    Thus for any $f\in C^\infty(\Mcal)$ supported outside $\overline{\bigcup\limits_{t\geq 0}\varphi^t(\supp W)}$, $s\geq 0$ we have $$ f e^{-ish^{-1}(-ihX -i(Q_\infty+q_1+W))}W g  =0.$$
    \item By a dual argument, for any $f\in C^\infty(\Mcal)$ supported outside $\overline{\bigcup\limits_{t\leq 0}\varphi^t(\supp W)}$, $t\geq 0$ we have $$ g W e^{-ith^{-1}(-ihX -i(Q_\infty+q_1))} f =0.$$
    One can also see this directly, as we can solve directly for $x\in\Sigma$,
    \begin{align*}
        w(t,x):=e^{-h^{-1}\int_0^t q_1(\varphi^{-s}(x))ds}f(\varphi^{-t}(x))=e^{-ith^{-1}(-ihX -i(Q_\infty+q_1))} f(x).
    \end{align*}
\end{itemize}

Now Lemma \ref{l:convexity}, the intersection of \eqref{e:supp_difference} with diagonal $\Delta=\{(x,x):x\in\Mcal\}$ 
\begin{align*}
    \overline{\bigcup\limits_{t\geq 0}\varphi^t(\supp W)\times \bigcup\limits_{t\leq 0}\varphi^t(\supp W)}\cap \Delta = \overline{\bigcup\limits_{t\geq 0}\varphi^t(\supp W)}\cap \overline{\bigcup\limits_{t\leq 0}\varphi^t(\supp W)}.
\end{align*}
is a compact set inside $\Ucal$. If we take $\chi=1$ near $\overline{\bigcup\limits_{t\geq 0}\varphi^t(\supp W)\times \bigcup\limits_{t\leq 0}\varphi^t(\supp W)}\cap \Delta$, then we have \eqref{e:trace_nochi}.

Then we need to estimate $\Tr (R_h^0(z)-\tilde{R}_h(z))$. Similar to \cite[\S 4.1]{local}, we write
\begin{align*}
    \Tr (R_h^0(z)-\tilde{R}_h(z))=-\frac{d}{dz}\log\det F(z)
\end{align*}
where 
$$F(z)=I+iW\tilde{R}_h(z),\quad F(z)^{-1}=I-iWR_h^0(z).$$
We have
\begin{align*}
    |\det F(z)|\leq (Ch^{-1})^{Ch^{-n}}\leq Ce^{Ch^{-n-1}}
\end{align*}
for $z\in[-h^\epsilon,h^\epsilon]+i[-C_1h,1]$ and
\begin{align*}
    |\det F(z)^{-1}|\leq  (Ch^{-1})^{Ch^{-n}}\leq Ce^{Ch^{-n-1}}
\end{align*}
for $z\in[-h^\epsilon,h^\epsilon]+i[C_2h,1]$ when $C_2>0$ is chosen to be large so that $R_h^0(z)$ exists and is bounded by $Ch^{-1}$.
We recall a lemma of Cartan (see \cite[Section 11.3 Theorem 4]{entire}).
\begin{lem}
Suppose $R>0$ and $g(z)$ is a holomorphic function in $D(z_0,2eR)$ with $g(z_0)=1$, then for any $\eta>0$ there exists a union of discs $\Dcal$ with sum of radii less that $\eta R$, such that
\begin{align*}
    \log |g(z)|\geq -\log (15e^3/\eta)\log\max\limits_{z\in D(z_0,2eR)}|g(z)|,\quad z\in D(z_0,R)\setminus \Dcal.
\end{align*}
Moreover, if $g(z)$ has $M$ zeros inside $D(z_0,2R)$, we may take $\Dcal$ to be the union of $M$ discs centered at the zeros.
\end{lem}
It follows that if we can avoid the union of discs, we have
\begin{align*}
    |\log\det F(z)|=\Ocal( h^{-n-1}).
\end{align*}
Since there are $\Ocal(h^{-n})$ zeros in the region $[h^\epsilon-h,h^\epsilon+h]+i[-C_1h,1]$, there are gaps of size $\sim h^{n+1}$. By Cauchy's formula we get
\begin{align*}
    \left|\frac{d}{dz}\log \det F(z)\right|=\Ocal(h^{-2n-2}).
\end{align*}
We conclude that we may choose contours appropriately such that along the contour, \eqref{dzetah} holds. 
\end{proof}

\begin{lem}
Let $F_A$ be as in Theorem \ref{t:localtrace}, then for $\psi\in C_c^\infty((0,\infty))$
\begin{align}
\label{distri0}
    \langle F_A,\psi\rangle=\frac{i}{2\pi}\sum\limits_{k\in\ZZ}\int_{\tilde{\gamma}_k^3} \hat{\psi}(\lambda)\frac{d}{d\lambda}\log\zeta_1(\lambda)d\lambda-\sum\limits_{\mu_j\in\Res(P)\cap\tilde{\Omega},\Im\mu_j\leq -A}\hat{\psi}(\mu_j).
\end{align}
\end{lem}
\begin{proof}
This follows from
\begin{align*}
    \int_{\mathbb{R}+iB}\hat{\psi}(\lambda)\frac{d}{d\lambda}\log\zeta_1(\lambda)d\lambda=-2\pi i\left\langle \sum\limits_\gamma\frac{T_\gamma^\# \delta(t-T_\gamma)}{|\det(I-\mathcal{P}_\gamma)|},\psi\right\rangle
\end{align*}
and
\begin{align*}
    \sum\limits_{\mu\in\Res(P),\Im \mu>-A}\hat{\psi}(\mu)=\left\langle  \sum\limits_{\mu\in\Res(P),\Im \mu>-A}e^{-i\mu t},\psi \right\rangle.
\end{align*}
\end{proof}
We claim that $F_A$ can be extended to $\Scal'(\RR)$ such that $\hat{F}_A$ is holomorphic in $\Im \lambda<A-\epsilon$ and 
\begin{align*}
    |\hat{F}_A(\lambda)|=\Ocal(\langle \lambda\rangle^{2n+1+\epsilon}),\quad \Im\lambda<A-\epsilon.
\end{align*}
To justify the claim we deal with the two terms in \eqref{distri0} separately. Let
\begin{align*}
    \langle u_k, \psi\rangle =\int_{\tilde{\gamma}_k^3}\hat{\psi}(\lambda)\frac{d}{d\lambda}\log\zeta_1(\lambda)d\lambda,
\end{align*}
we notice that $u_k$ admits a natural extension to $\RR$ by
\begin{align*}
    u_k(t)=\mathbbm{1}_{t>0}\int_{\tilde{\gamma}^3_k}e^{-it\lambda}\frac{d}{d\lambda}\log\zeta_1(\lambda)d\lambda
\end{align*}
and its Fourier transform
\begin{align*}
    \hat{u}_k(\mu)=\int_{\tilde{\gamma}^3_k}\frac{1}{i(\lambda+\mu)}\frac{d}{d\lambda}\log\zeta_1(\lambda)d\lambda
\end{align*}
extends holomorphically to $\{\Im \mu<A-\epsilon\}$. If we take the sum $\sum\limits_{k\in\ZZ}\hat{u}_k$ naively, then it does not converge. Instead we choose polynomial
\begin{align*}
    p_k(\lambda,\mu)=\sum\limits_{j=0}^{2n+1} (-1)^j\lambda^{-j-1}\mu^j
\end{align*}
such that
\begin{align*}
    (\lambda+\mu)^{-1}-p_k(\lambda,\mu)=\frac{\mu^{2n+2}}{\lambda^{2n+2}(\lambda+\mu)}
\end{align*}
and thus
\begin{align*}
    \sum\limits_{k\in\ZZ}\left(\hat{u}_k-p_k(\lambda,\mu)\int_{\tilde{\gamma}^3_k}\frac{1}{i}\frac{d}{d\lambda}\log\zeta_1(\lambda)d\lambda\right)
\end{align*}
converges absolutely and uniformly. Let $u\in \Scal'(\RR)$ be the Schwartz distribution such that 
\begin{align*}
    \hat{u}:=\sum\limits_{k\in\ZZ}\left(\hat{u}_k-p_k(\lambda,\mu)\int_{\tilde{\gamma}^3_k}\frac{1}{i}\frac{d}{d\lambda}\log\zeta_1(\lambda)d\lambda\right),
\end{align*}
then $\hat{u}(\mu)$ is a holomorphic function in $\{\Im\mu<A-\epsilon\}$. Moreover, since
\begin{align*}
    \partial^{2n+2}\hat{u}=\sum\limits_{k\in\ZZ}\partial^{2n+2}\hat{u}_k
\end{align*}
converges, we have
\begin{align}
\label{distri1}
    \langle t^{2n+2}u,\varphi\rangle=\sum\limits_{k\in\ZZ}\int_{\tilde{\gamma}^3_k}\widehat{t^{2n+2}_+\varphi}(\lambda)\frac{d}{d\lambda}\log\zeta_1(\lambda) d\lambda
\end{align}
for $\varphi\in C_c^\infty(\RR)$. In other words, $u$ is an extension of the distribution defined by the right hand side of \eqref{distri1} from $(0,\infty)$ to $\RR$ such that $u=0$ in $(-\infty,0)$.  

Since 
\begin{align*}
    \hat{u}_k-p_k(\lambda,\mu)\int_{\tilde{\gamma}^3_k}\frac{1}{i}\frac{d}{d\lambda}\log\zeta_1(\lambda)d\lambda&=\int_{\tilde{\gamma}^3_k}\frac{\mu^{2n+2}}{i\lambda^{2n+2}(\lambda+\mu)}\frac{d}{d\lambda}\log\zeta_1(\lambda)d\lambda\\
    &=\Ocal\left(\frac{\mu^{2n+2}}{\lambda^{1-\epsilon}(\lambda+\mu)}\right),
\end{align*}
we get
\begin{align*}
    \hat{u}=\sum\limits_{k\in\ZZ}\Ocal\left(\frac{\mu^{2n+2}}{(1+|k|)^{1-\epsilon}(1+|k|+|\mu|)}\right)=\Ocal(\langle \mu\rangle^{2n+1+\epsilon'})
\end{align*}
for $\{\Im\mu<A-\epsilon\}.$ The second term in \eqref{distri0} can be handled similarly, with a better bound.

\subsection{Proof of the lower bound}\label{s:weaklowerbound}
We prove Theorem \ref{t:lower} by improving the method in \cite[Section 5]{local}.

Let $$N_A(r)=\#(\Res(P)\cap\{|\mu|\leq r,\Im\mu>-A\}),$$ 
we would like to show $N_A(r)\geq r^\delta/C$.
Let $\varphi\in C_c^\infty(\RR)$ such that
\begin{align*}
    \varphi\geq 0,\quad \varphi(0)=1,\quad\supp\varphi\subset [-1,1]
\end{align*}
and $\varphi_{\ell,d}(x)=\varphi((x-d)/\ell)$ for $0<\ell<1<d$. We apply local trace formula and find
\begin{align}\label{e:trace-lower-1}
    \sum\limits_{\mu\in \Res(P),\Im\mu>-A}\hat{\varphi}_{\ell,d}(\mu)+\langle F_A,\varphi_{\ell,d}\rangle=\sum\limits_\gamma \frac{T_\gamma^\#\varphi_{\ell,d}(T_\gamma)}{|\det(I-\mathcal{P}_\gamma)|}.
\end{align}
Since $\supp \varphi \subset [-1,1]$, we have
\begin{align*}
    |\hat{\varphi}_{\ell,d}(\zeta)|=|\ell\hat{\varphi}(\ell\zeta)e^{-id\zeta}|\leq C\ell e^{(d-\ell)\Im\zeta}(1+|\ell\zeta|)^{-N}.
\end{align*}
For the first term in \eqref{e:trace-lower-1} we have
\begin{align*}
    \sum\limits_{\mu\in \Res(P),\Im\mu>-A}\hat{\varphi}_{\ell,d}(\mu)&\leq C\ell\int_{0}^\infty (1+\ell r)^{-N}dN_A(r)\\
    &= CN \ell\int_0^\infty (1+r)^{-N-1}N_A\left(\frac{r}{\ell}\right)dr.
\end{align*}
Let $\Gamma=\cup_{k\in\ZZ}\tilde{\gamma}_k^3$, then the second term in \eqref{e:trace-lower-1} can be estimated as follows.
\begin{align*}
    &\langle F_A,\varphi_{\ell,d}\rangle=\frac{i}{2\pi}\int_{\Gamma}\hat{\varphi}_{\ell,d}(\lambda)\frac{d}{d\lambda}\log\zeta_1(\lambda)d\lambda-\sum\limits_{\mu_j\in\Res(P)\cap\tilde{\Omega},\Im\mu_j\leq -A}\hat{\varphi}_{\ell,d}(\mu_j)\\
    &\lesssim\int_\Gamma \ell e^{(d-\ell)\Im\zeta}(1+|\ell\zeta|)^{-N}\langle \zeta\rangle^{2n+2}d\zeta+\sum\limits_{\mu_j\in\Res(P)\cap\tilde{\Omega},\Im\mu_j\leq -A} \ell e^{(d-\ell)\Im \mu_j}(1+\ell|\mu_j|)^{-N}\\
    &=\Ocal(\ell^{-2n-2}e^{(d-\ell)(-A+\epsilon)}).
\end{align*}
Let $\gamma_0$ be a primitive closed orbit, $d=kT_{\gamma_0}$, then for some $\alpha=\alpha(\gamma_0)>0$ we have
\begin{align*}
    \sum\limits_\gamma \frac{T_\gamma^\#\varphi_{\ell,d}(T_\gamma)}{|\det(I-\mathcal{P}_\gamma)|}\geq \frac{T_{\gamma_0}\varphi_{\ell,d}(kT_{\gamma_0})}{|\det(I-\mathcal{P}_{k\gamma_0})|}\geq Ce^{-\alpha d}.
\end{align*}\
Thus from \eqref{e:trace-lower-1} we get the following inequality
\begin{align}\label{e:trace-lower-2}
    e^{-\alpha d}\lesssim \ell\int_0^\infty (1+r)^{-N-1}N_A\left(\frac{r}{\ell}\right)dr+\ell^{-2n-2}e^{(d-\ell)(-A+\epsilon)}.
\end{align}
Separating the first integral as $\int_0^{\ell^{-\epsilon_0}}+\int_{\ell^{-\epsilon_0}}^\infty$ and note
\begin{align*}
     \int_{\ell^{-\epsilon_0}}^\infty (1+r)^{-N-1}N_A\left(\frac{r}{\ell}\right)dr\lesssim \int_{\ell^{-\epsilon_0}}^\infty (1+r)^{-N-1}(r/\ell)^{n+1}dr\lesssim \ell^{\epsilon_0(N-n-1)-n-1},
\end{align*}
we can take $N=(1+\epsilon_0^{-1})(n+1)$ and conclude from \eqref{e:trace-lower-2} that
\begin{equation}\label{e:trace-lower-3}
    e^{-\alpha d}\lesssim \ell N_A(\ell^{-\epsilon_0-1})+\ell+\ell^{-2n-2}e^{(d-\ell)(-A+\epsilon)}.
\end{equation}

Pick $\ell=e^{-M d}$, $M>\alpha$ and $A=(2n+2)M+\alpha+1$, we have 
\begin{equation*}
    \ell^{-2n-2}e^{(d-\ell)(-A+\epsilon)}\lesssim e^{-(\alpha+1-\epsilon)d}\ll e^{-\alpha d},\quad \ell\ll e^{-\alpha d}
\end{equation*}
as $d\to \infty$. So we conclude from \eqref{e:trace-lower-3} that for $\ell=e^{-kMT_{\gamma_0}}$, $k\in\mathbb{N}$,
\begin{equation*}
    N_A(\ell^{-\epsilon_0-1})\gtrsim \ell^{-1} e^{-\alpha d}=\ell^{-1+\alpha M^{-1}},
\end{equation*}
This implies for any $r\gg 1$,
\begin{equation*}
    N_A(r)\gtrsim r^{(1-\alpha M^{-1})(1+\epsilon_0)^{-1}}.
\end{equation*}
Taking $M$ large enough and $\epsilon_0$ small enough so that
$(1-\alpha M^{-1})(1+\epsilon_0)^{-1}>\delta$ (for example, $M=10\alpha(1-\delta)^{-1}$, $\epsilon_0=(1-\delta)/10$), we conclude Theorem~\ref{t:lower}.
\section{Axiom A flows with transversality condition}
\label{s:axioma}
In this section, we discuss the generalization of Theorem \ref{t:upper} and \ref{t:lower} to Axiom A flows. In this way, we need a transversal condition to assure the meromorphic continuation, see \cite{meddane}. This condition is natural for the construction of suitable weight functions for the anisotropic Sobolev spaces.

\subsection{Review of Axiom A flows}
\label{s:axiomabasic}
We first briefly review the basic dynamical setup of Axiom A flows: Let $M$ be a compact $C^\infty$ Riemannian manifold without boundary, $n=\dim M$ and $\varphi^t=\exp(tX)$ be a $C^\infty$ flow generated by a vector field $X$. We recall the following definitions, see \cite{smale} for details:
\begin{defi}
We call $x\in M$ a nonwandering point if for every neighbourhood $V$ of $x$ and every $T>0$, there exists $t\in\mathbb{R}$ such that $|t|\geq T$ and $\varphi^t(V)\cap V\neq\varnothing$. The set consisting of all nonwandering points is called the nonwandering set.
\end{defi}
\begin{defi}
The flow $\varphi^t$ is called Axiom A if
\begin{itemize}
    \item the nonwandering set is the disjoint union of the set of all fixed points $\mathcal{F}$ and the closure of the set of all closed orbits $\mathcal{K}$;
    \item all fixed points are hyperbolic for $\varphi^t$;
    \item $\mathcal{K}$ is hyperbolic for $\varphi^t$.
\end{itemize}
\end{defi}

To study the Axiom A flow, we decompose it to several basic sets as follows.
\begin{defi}
We say that a compact $\varphi^t$-invariant set $K\subset M$ is locally maximal if there exists a neighbouthood $V$ of $K$ such that
\begin{align*}
    K=\bigcap\limits_{t\in\mathbb{R}}\varphi^t(V).
\end{align*}
\end{defi}
\begin{defi}
A compact $\varphi^t$-invariant set $K\subset M$ is called a basic hyperbolic set if
\begin{itemize}
    \item $K$ is locally maximal;
    \item $K$ is hyperbolic;
    \item the flow $\varphi^t|_K$ is topologically transitive;
    \item $K$ is the closure of the union of all closed orbits of $\varphi^t|_K$.
\end{itemize}
We also allow a basic hyperbolic set to be the set of a single hyperbolic fixed point.
\end{defi}
We recall the following lemmas from \cite{smale}.
\begin{lem}
Assume that $\varphi^t$ is an Axiom A flow and let $\mathcal{F}$ be the set of fixed points, $\mathcal{K}$ be the closure of the union of all closed orbits, then we can write the nonwandering set $\mathcal{F}\sqcup\mathcal{K}$ as a finite disjoint union
\begin{equation}
\label{e:axioma-decomposition}
    \mathcal{F}\sqcup\mathcal{K}=K_1\sqcup K_2\sqcup \cdots \sqcup K_N
\end{equation}
where each $K_j$ is a basic hyperbolic set.
\end{lem}

Now we state the strong transversality condition. 
\begin{itemize}
\item[(T)] For any pairs $(x_-,x_+)\in K_i\times K_j$ and every $x\in W^u(x_-)\cap W^s(x_+)$,
$$T_xW^u(x_-)+T_xW^{so}(x_+)=T_xW^{uo}(x_-)+T_xW^s(x_+)=T_xM.$$
\end{itemize}
Here $W^u$, $W^s$, $W^{uo}$, $W^{so}$ are the unstable, stable, weakly unstable, weakly stable manifolds, respectively, see \cite[\S 17.4]{dynamic} or \cite{dyatlovnotes}. Under this assumption, Meddane \cite{meddane} proves the meromorphic continuation of the resolvent and thus defines the Pollicott--Ruelle resonances. We still denote $P=-iX$, and the set $\Res(P)$ of Pollicott--Ruelle resonances are the poles of the meromorphic continuation of $R(\lambda)=(P-\lambda)^{-1}:C^\infty(M)\to\mathcal{D}'(M)$ from $\Im \lambda\gg1$ to $\mathbb{C}$. One of the key ingredient is the following result on unrevisited set.

\begin{lem}(\cite[Proposition 2.3]{meddane})
\label{l:unrevisited}
Let $K$ be a basic hyperbolic set, and $\varphi^t$ satisfies the strong transversal condition (T). Then there exists arbitrarily small unrevisited neighbourhood $\mathcal{V}$ of $K$, i.e. for any $m\in\mathbb{N}$,
$$x,\varphi^m(x)\in\mathcal{V}\quad\Rightarrow\quad \forall k=0,1,\ldots,m, \varphi^k(x)\in\mathcal{V}.$$
\end{lem}

\subsection{From open hyperbolic systems to Axiom A flows}
The key to pass from open hyperbolic systems to Axiom A flows is the following lemma \cite[Lemma 3.2]{axioma}, which states that near any basic hyperbolic set $K_j$, the flow can be modified into an open hyperbolic system. Strictly speaking, \cite[Lemma 3.2]{axioma} only deals with the case for $\mathcal{K}$, the case of fix points $\Fcal$ is simpler and we review in \S \ref{s:fixedpoints}.
\begin{lem}
\label{l:basic-to-open}
Let $K\subset M$ be a locally maximal hyperbolic set for $\varphi^t$. Then there exists a neighbourhood $\mathcal{U}$ of $K$ with $C^\infty$ boundary and a $C^\infty$ vector field $X_0$ on $\overline{\mathcal{U}}$ such that 
\begin{itemize}
    \item The boundary of $\mathcal{U}$ is strictly convex in the sense that 
    \begin{align}
    \label{convex}
        X_0\rho(x)=0\Rightarrow X_0^2\rho(x)<0
    \end{align}
    for a boundary defining function, i.e. $\sgn \rho(x)=\mathbbm{1}_{\mathcal{U}}-\mathbbm{1}_{\overline{\mathcal{U}}^c}$ with $d\rho\neq 0$ on $\partial \mathcal{U}$;
    \item $X_0=X$ in a neighbourhood of $K$;
    \item $K=\bigcap\limits_{t\in\mathbb{R}}\varphi_0^t(\mathcal{U})$ where $\varphi_0^t=\exp(tX_0)$.
\end{itemize}
In particular, $(\overline{\mathcal{U}},X_0)$ is an open hyperbolic system with trapped set $K$.
\end{lem}

Using this lemma, we can prove that the Pollicott--Ruelle resonances for the Axiom A flow are exactly those resonances for the open hyperbolic systems obtained from each nontrivial basic hyperbolic set.

\begin{prop}
Let $\varphi^t=\exp(tX)$ be an Axiom A flow on $M$ satisfying the transversality condition (T),  $P=-iX$ and $R(\lambda)=(P-\lambda)^{-1}:C^\infty(M)\to\mathcal{D}'(M)$ be the resolvent. Let $K$ be a basic hyperbolic set and $(\overline{\mathcal{U}}_K,X_K)$ the open hyperbolic system as in Lemma \ref{l:basic-to-open}, extended to $M$ as in Section \ref{s:open}. Let $R_K(\lambda)=(P_K-\lambda)^{-1}:C^\infty(M)\to\mathcal{D}'(M)$ be the resolvent with $P_K=-iX_K$, then there exists a function $\chi_K\in C_c^\infty(\mathcal{U}_K)$ supported in a sufficiently small neighbourhood of $K$ and equals 1 near $K$ such that for all $\Im\lambda\gg 1$,
\begin{align}
\label{e:res-opentoaxiom}
    \chi_K R(\lambda)\chi_K =\chi_K R_K(\lambda)\chi_K: C^\infty(M)\to\mathcal{D}'(M)
\end{align}
\end{prop}
\begin{proof}
Let $\varphi_K^t=e^{tX_K}$ be the flow generated by $X_K$. We only need to prove that for all $t\in\mathbb{R}$,
\begin{align}
\label{e:flow-opentoaxiom}
    \chi_K\varphi^{-t*}\chi_K=\chi_K\varphi_K^{-t*}\chi_K:C^\infty(M)\to C^\infty(M),
\end{align}
which means if $x\in\supp\chi_K$, $t\in\mathbb{R}$, then $\varphi^t(x)\in\supp\chi_K$ if and only if $\varphi^t_K(x)\in\supp\chi_K$ and in this case $\varphi^t(x)=\varphi^t_K(x)$.  Given this, we can use the following relation
$$R(\lambda)=(-iX-\lambda)^{-1}=i\int_0^\infty e^{it\lambda}\varphi^{-t}dt$$
and the analogue for $R_K$ and $X_K$ for $\Im\lambda\gg1$ to obtain \eqref{e:res-opentoaxiom}.  

Let us fix a smaller neighbourhood $\mathcal{V}\subset \Ucal_K$ such that $X=X_K$ on $\mathcal{V}$.
By Lemma \ref{l:unrevisited}, we may choose an unrevisited neighbourhood $\mathcal{V}'$ such that $\varphi^t(\mathcal{V}')\subset\mathcal{V}$ for $t\in[-1,1]$. Now we choose $\chi_K\in C_c^\infty(\Ucal_K)$ such that
\begin{itemize}
\item[(a)] $\varphi^t(\supp\chi_K)\subset \mathcal{V}'$ for $-1\leq t\leq 1$,
\item[(b)] $\varphi_K^t(\supp\chi_K)\subset \mathcal{U}_K$ for $-T\leq t\leq T$. Here $T>0$ is chosen as in Proposition \ref{p:property-K}(2)  such that $\varphi_K^T(x),\varphi_K^{-T}(x)\in\overline{\Ucal}_K$ implies that $x\in\mathcal{V}$. 
\end{itemize}

On one hand, if $x,\varphi^t(x)\in \supp\chi_K$, let $m$ be the integer with the same sign as $t\in\mathbb{R}$ such that $|m|$ is the smallest integer larger or equal to $|t|$. Then $\varphi^m(x)\in\mathcal{V}'$ by (a). Since $\mathcal{V}'$ is unrevisited for $\varphi^t$, we know $\varphi^k(x)\in\mathcal{V}'$ for every integer $k$ between 0 and $m$. Thus by the construction of $\mathcal{V}'$, $\varphi^s(x)\in\mathcal{V}$ for every $s$ between 0 and $t$. In particular, $X$ and $X_K$ agrees on $\varphi^s(x)$ for $s$ between 0 and $t$ and thus $\varphi^t(x)=\varphi^t_K(x)$.

On the other hand, if $x,\varphi_K^t(x)\in\supp\chi_K$, then by (b), $\varphi^{\pm T}(x),\varphi^{t\pm T}(x)\in\mathcal{U}_K$. Now the convexity \eqref{c:convexity-flow} of $\mathcal{U}_K$ with respect to $\varphi^t_K$ shows that for any $s$ between 0 and $t$, we have $\varphi_K^{s\pm T}(x)\in\mathcal{U}_K$ and thus by Proposition \ref{p:property-K}(2), $\varphi_K^s(x)\in \mathcal{V}$. As above, $X$ and $X_K$ agrees on $\varphi_K^s(x)$ for $s$ between 0 and $t$ and thus $\varphi^t(x)=\varphi^t_K(x)$. This justifies \eqref{e:flow-opentoaxiom} and thus proves \eqref{e:res-opentoaxiom}.
\end{proof}


\subsection{The case with a fixed point}\label{s:fixedpoints}

\cite{open} does not deal with the case of a hyperbolic fixed point $x_0$ technically. In this section, we study the dynamics near the fixed point. Let $e^{tA}=d\varphi^t|_{x_0}$ where $A:T_{x_0}M\to T_{x_0}M$. We may choose the coordinate so that $A$ is in real Jordan normal form, i.e. 
\begin{align*}
    A={\rm diag}\left(A_1,\cdots, A_k\right),\quad A_k=\begin{pmatrix}B_k&1&0&\cdots&0\\
    0&B_k&1&\cdots&0\\
    \vdots&\vdots&\vdots&\ddots&\vdots\\
    0&0&0&\cdots&B_k\end{pmatrix},
\end{align*}
where
\begin{align*}
B_k=\lambda_k\in\RR \text{ or } \begin{pmatrix}a_k&-b_k\\
    b_k&a_k\end{pmatrix}\in M_2(\RR).
\end{align*}
The condition that $x_0$ is hyperbolic means $\lambda_k\neq 0$ and $a_k\neq 0$. By conjugation (or equivalently, choosing a different coordinate), we may assume for some $0<\epsilon\ll 1$,
\begin{align*}
    A={\rm diag}\left(A_1,\cdots, A_k\right),\quad A_k=\begin{pmatrix}B_k&\epsilon&0&\cdots&0\\
    0&B_k&\epsilon&\cdots&0\\
    \vdots&\vdots&\vdots&\ddots&\vdots\\
    0&0&0&\cdots&B_k\end{pmatrix}.
\end{align*}
Suppose $x_0=0$ in the coordinate and let $\Ucal=B_r(0)$ for $r\ll 1$ and $\rho(x)=r^2-|x|^2$, then $X=Ax\cdot\partial_x+\Ocal(|x|^2)\partial_x$ and
\begin{align*}
    X^2\rho(x)=-2Ax\cdot\partial_x((Ax)\cdot x) +\Ocal(|x|^3)=-2(Ax)\cdot(A+A^T)x+\Ocal(|x|^3).
\end{align*}
Since $\Re(A_k+A_k^T)A_k=2\lambda_k^2+\Ocal(\epsilon)>0$ or $2a_k^2+\Ocal(\epsilon)>0$ for $\epsilon\ll 1$, we conclude $X^2\rho(x)=-2\Re((A+A^T)Ax\cdot x)+\Ocal(|x|^3)<0$ for all $0\neq x\ll 1$.

Now we claim $\{x_0\}$ is the trapped set $\{x_0\}=\cap_{t\in\RR} \varphi^t(\overline{\Ucal})$. If $x\in \cap_{t\in\RR} \varphi^t(\overline{\Ucal})$, then $x_t=\varphi^t (x)\in \overline{\Ucal}$ and $\frac{d}{dt}\rho(x_t)=X\rho(x_t)$. Since $\rho(x_t)$ is a concave function, if $X\rho(x_t)<0$ for some $t$, then $x_t\notin \overline{\Ucal}$ as $t\to \infty$. If $X\rho(x_t)>0$, then $x_t\notin\overline{\Ucal}$ as $t\to -\infty$. So $X\rho(x_t)$ has to be identically zero and thus $X^2\rho(x_t)=0$, which implies $x=x_0$. 

Now we have an open system $\Ucal$ such that
\begin{itemize}
    \item $\overline{\Ucal}$ is an $n$-dimensional smooth manifold with interior $\Ucal$ and boundary $\partial\Ucal$, $X$ is a smooth vector field vanishing only at $x_0$, and $\varphi^t=e^{tX}$;
    \item There is a boundary defining function $\rho:\overline{\Ucal}\to \RR$ such that $X^2\rho(x)\leq - |x|^2/C$;
    \item $\frac{d}{dt}|_{t=0}d\varphi^t:T_{x_0}\Ucal\to T_{x_0}\Ucal$ is a hyperbolic map.
\end{itemize}

As in \cite[Lemma 1.1]{open}, we can extend $\Ucal$ to a closed manifold $\Mcal$ such that $\Ucal$ is convex in the sense that
\begin{align}
    x,\varphi^T(x)\in \Ucal, \quad T\geq 0\Rightarrow \varphi^t(x)\in\Ucal,\quad\forall 0\leq t\leq T.
\end{align}
Moreover, the boundary  defining function $\rho$ can be extended to $\Mcal$ such that $X$ vanishes on $\{\rho=-\epsilon\}$ and
\begin{align*}
    |\rho(x)|\leq 2\epsilon, \rho(x)\neq -\epsilon, X\rho(x)=0 \Longrightarrow X^2\rho(x)<0.
\end{align*}
Let $\Gamma_\pm=\bigcap\limits_{\pm t\geq 0}\varphi^t(\overline{\Ucal})$, we have vector bundles $E_\pm^* $ over $\Gamma_\pm$ that extends $E_u^*$ and $E_s^*$ over $x_0$, defined (e.g. for $E_+^*$) as
\begin{align*}
    E_{+}^*:=\lim\limits_{t\to -\infty}(d\varphi^t(x))^{T}\tau_{x_0\to \varphi^t(x)}E_u^*(x_0)
\end{align*}
where $\tau$ is the parallel transport with respect to an arbitrary pre-fixed metric. As in \cite[Lemma 1.10]{open}, this converges to a continuous vector bundle over $\Gamma_+$ extending $E_u^*$.

It follows from the definition that
\begin{itemize}
    \item For $x\in \Gamma_\pm$, $\varphi^t(x)\to x_0$ as $t\to \mp\infty$;
    \item For $x\in\Gamma_\pm$, $\xi\in E_\pm^*(x)$, there exist $C,\theta>0$ such that $|d\varphi^t_x \xi|\leq Ce^{-\theta|t|}|\xi|$, $t \to \mp\infty$;
    \item For $x\in\Gamma_\pm$, $\xi\in T_x^*\Mcal\setminus E_\pm^*(x)$, as $t\to \mp\infty$, $|(d\varphi^{t}(x))^{-T}\xi|\to \infty$  and 
$$\frac{(d\varphi^{ t}(x))^{-T}\xi}{|(d\varphi^{t}(x))^{-T}\xi|}\to E_\mp^*|_{x_0}.$$
\end{itemize}

From this, we can construct a weight function $m\in C^\infty(S^*\Mcal;\RR)$ following \cite[Lemma 3.1]{open} such that
\begin{itemize}
    \item $m=\mp 1$ in a neighbourhood of $\kappa(E_\pm^*)$;
    \item $H_pm\leq 0$;
    \item $\supp m\subset\{\rho>-2\epsilon\}$ and $\supp m\cap \{\rho=-\epsilon\}\cap \{X_1\rho=0\}=\varnothing$.
\end{itemize}
The function $m$ is of the form 
\begin{align}\label{eqn:fix_pt_weight}
    m=m_- -m_+ +R(\chi_-\circ\pi-\chi_+\circ \pi)
\end{align}
where $m_\pm$ is constructed as in \cite[Lemma 1.12]{open} and $\chi_\pm$ is constructed as in \cite[Lemma 1.8]{open}. $R$ is taken to be a large number.
\begin{rem}
Intuitively, $m=-(A\xi)\cdot \xi$ would satisfy $H_pm\leq 0$ with correct signs near $\kappa(E_\pm^*)$. But it does not have compact support so we need to define it as in \eqref{eqn:fix_pt_weight}.
\end{rem}

The proof in \cite{open} and \S \ref{s:upper} goes through. In particular, since $\{p=1\}\cap T_{x_0}^*\Ucal=\varnothing$, we do not need the absorbing potential $W$ in Theorem \ref{t:resolvent}. We have
\begin{prop}
Suppose we have an open system with a single fixed point. Let $\beta>0$, then there exists $h_0>0$ such that for $0<h<h_0$ and $z\in [1-h,1+h]+i[-\beta h,1]$, the operator $P_h(z)=-ihX-i(Q_\infty+q_1)\in\Psi^1_h$ is invertible on $\Hcal^s_h$, with
\begin{align*}
    \|P_h(z)^{-1}\|_{\Hcal^s_h\to \Hcal^s_h}\leq \frac{C}{h}.
\end{align*}
In particular, an open system with a single fixed point has finitely many resonances inside a strip $\{\lambda\in\CC:\Im\lambda>-\beta\}$.
\end{prop}


\subsection{Proof of the main theorems}
Now we show that the resonances for an Axiom A flow are exactly the union of the resonances for each open hyperbolic system and thus prove the main theorems. To be more precise,
let $\varphi^t=\exp(tX)$ be an Axiom A flow on $M$ satisfying the transversality condition (T),  $P=-iX$ and $R(\lambda)=(P-\lambda)^{-1}$ be the resolvent. The spectral decomposition is given by \eqref{e:axioma-decomposition}. For each basic hyperbolic set $K_j$, fix an open hyperbolic system $(\overline{\mathcal{U}}_j,X_j)$ as in Lemma \ref{l:basic-to-open}, extended to $M$ as in Section \ref{s:open} and this gives the resolvent $R_j(\lambda)=(P_j-\lambda)^{-1}$ where $P_j=-iX_j$. Then we can choose cutoff function $\chi_j\in C_c^\infty(\mathcal{U}_j)$ supported in a sufficiently small neighbourhood of $K_j$ and equals 1 near $K_j$ such that
\begin{align}
\label{axioma=open_1}
    \chi_j R(\lambda)\chi_j =\chi_j R_j(\lambda)\chi_j
\end{align}
where $R_j(\lambda)$ is the resolvent for the open system $(\overline{\Ucal}_j,X_j)$. We can also assume that for $j\neq j'$, $\supp\chi_j\cap\supp\chi_{j'}=\varnothing$. We define the corresponding spaces of (generalized) resonances states:
$$\Res_X(\lambda):=\ker(P-\lambda)^{m_X(\lambda)}=\range\Pi_\lambda$$
where $\lambda\in\Res(P)$ and $m_X(\lambda)$ is the multiplicity of $\lambda$, given by
$m_X(\lambda)=\rank\Pi_\lambda$
where
$$\Pi_\lambda=\frac{1}{2\pi i}\oint_\lambda R(z)dz:C^\infty(M)\to\mathcal{D}'(M)$$
is the projection onto the space of generalized resonance states. Here the integral is taken over a sufficiently small positively oriented circle around $\lambda$ that does not enclose any other resonances. Similarly, 
$$\Res_{X_j}(\lambda):=\ker(P_j-\lambda)^{m_{X_j}(\lambda)}=\range(\Pi_{j,\lambda})$$
where $\lambda\in\Res(P_j)$ and $m_{X_j}(\lambda)$ is the multiplicity of $\lambda$, given by
$m_{X_j}(\lambda)=\rank\Pi_{j,\lambda}$
where
$$\Pi_{j,\lambda}=\frac{1}{2\pi i}\oint_\lambda R_j(z)dz:C^\infty(\Ucal_j)\to\mathcal{D}'(\Ucal_j).$$

As in \cite[\S 2.5.2]{meddane}, we fix a total order relation on the family of basic sets so that $K_i\leq K_j$ implies $i\leq j$. 

\begin{prop}
\label{p:res-axioma}
We have an isomorphism between generalized resonant states
\begin{align}
    \Res_X(\lambda)\cong \bigoplus\limits_{j=1}^N \Res_{X_j}(\lambda),
\end{align}
where the isomorphism is defined by 
\begin{align}
\label{e:axiomatoopen}
    u\mapsto (u_j:=\Pi_{j,\lambda}\chi_ju).
\end{align}
\end{prop}

\begin{proof}
Let $E_u^*$ be the unstable bundle on $M$ as in \cite[Definition 2.3]{meddane}. By the construction of the weight function in \cite[Proposition 3.4]{meddane}, we may conclude $\WF(u)\subset E_u^*$ for any $u\in \Res_X(\lambda)$. By the wavefront condition \cite[(0.14)]{open}, the map \eqref{e:axiomatoopen} is well-defined.

Now we prove \eqref{e:axiomatoopen} is injective. Suppose $0\neq u\in \Res_X(\lambda)$ maps to zero in \eqref{e:axiomatoopen}. 
Let $K_j$ be the minimal basic set with $\supp u\cap K_j\neq\varnothing$. Since $u_j=0$, by \eqref{axioma=open_1},
\begin{equation}
\label{e:projection-chichi}
    \chi_j\Pi_\lambda\chi_j u=\chi_j\Pi_{j,\lambda}\chi_j u =\chi_j u_j=0.
\end{equation}
Let $\tilde{\chi}_j\in C_c^\infty(\Ucal_j)$ such that 
\begin{itemize}
    \item[(a)] $\tilde{\chi}_j=1$ near $K_j$ and $\supp\tilde{\chi}_j\subset \mathcal{W} $ where $\mathcal{W}$ is a small neighborhood of $K_j$ such that $\varphi^t(\mathcal{W})\subset\mathcal{W}'$ for $t\in [-1,1]$, and $\mathcal{W}'$ is an unrevisited neighbourhood of $K_j$ such that $\mathcal{W}'\cap \supp (1-\chi_j)=\varnothing$.
    \item[(b)] $\supp (\tilde{\chi}_j)\cap \bigcup\limits_{l> j} W^{u}(K_l)=\varnothing $.
\end{itemize}
We claim that 
\begin{equation}
\label{e:projection-chi}
\tilde{\chi}_j\Pi_\lambda(1-\chi_j)u=0.
\end{equation}
It suffices to show $\tilde{\chi}_j \varphi^{-t*}(1-\chi_j)u=0$ for $t\geq 0$, then we can integrate over $t\in[0,\infty)$ to obtain $\tilde{\chi}_jR(\lambda)(1-\chi_j)u=0$ for $\Im\lambda\gg1$, and thus for all $\lambda\in\mathbb{C}$ by analytic continuation.

If $x\in \supp (\tilde{\chi}_j \varphi^{-t*}(1-\chi_j)u)$, then $\varphi^{-t}(x)\in \supp u$. There exists $\ell\in\{1,\ldots,N\}$ such that $\varphi^{-T}(x)\to K_l$ as $T\to\infty$. We must have $l\geq j$ since $(X+\lambda)^{m_X(\lambda)}u=0$ gives $\supp u\cap K_\ell\neq\varnothing$. In addition, since $x\in\supp\tilde{\chi}_j$, we must have $l=j$ by (b). But then by (a), we see $x\in \supp \tilde{\chi}_j\subset \mathcal{W}$ and $\varphi^{-T}(x)\in \mathcal{W}$ for any large enough $T>0$, then $\varphi^{-t}(x)\in \mathcal{W}'$ for any $t\geq0$, and in particular $\varphi^{-t}(x)\not\in\supp (1-\chi_j)$ for any $t\geq0$. Thus $\tilde{\chi}_j \varphi^{-t*}(1-\chi_j)u=0$ for $t\geq 0$ and \eqref{e:projection-chi} follows. From \eqref{e:projection-chichi} and \eqref{e:projection-chi}, and by (a), $\tilde{\chi_j}=\tilde{\chi}_j\chi_j$, we can then conclude 
\begin{align*}
    \tilde{\chi}_j u=\tilde{\chi}_j\Pi_\lambda u=\tilde{\chi}_j\Pi_\lambda(1-\chi_j)u+\tilde{\chi}_j\Pi_\lambda \chi_j u =0, 
\end{align*}
contradictory to the assumption that $\supp u\cap K_j\neq\varnothing$. Thus we conclude \eqref{e:axiomatoopen} is injective.

Finally we prove that \eqref{e:axiomatoopen} is a bijection. We only need to show both sides have the same dimension. Let us fix a partition of unity 
\begin{align*}
    1=\chi+\sum\limits_j \chi_j
\end{align*}
where $\chi$ is supported outside all basic sets and $\chi_j\in C_c^\infty(\Ucal_j)$.
Then we have
\begin{align*}
    \Pi_\lambda=\chi\Pi_\lambda\chi+\sum\limits_j \chi_j\Pi_\lambda\chi_j+\sum\limits_{j\neq j'}\chi_{j}\Pi_\lambda\chi_{j'}+\sum\limits_{j}(\chi \Pi_\lambda\chi_j  +\chi_j\Pi_\lambda\chi)
\end{align*}
Here
\begin{itemize}
\item $\chi \Pi_\lambda\chi=0$: since $\supp\chi\cap\bigcup_jK_j=\varnothing$, there exists $T_0>0$ such that for $t>T_0$, $\chi\varphi^{-t}\chi=0$ and thus
$$\chi R(\lambda)\chi=i\int_0^T\chi e^{it\lambda}\varphi^{-t}\chi dt$$
is holomorphic on $\mathbb{C}$.
\item $\Tr^\flat(\chi_j\Pi_\lambda\chi_{j'})=0$ since $\supp\chi_j\cap\supp\chi_{j'}=\varnothing$.
\item For the last term, we choose $\tilde{\chi}_j\in C_c^\infty(\mathcal{U}_j)$ such that $\tilde{\chi}_j=1$ near $K_j$ and $\supp\chi\cap\supp\tilde{\chi}_j=\varnothing$. Since $\supp\chi_j(1-\tilde{\chi}_j)$ does not intersect any basic set, the same argument as above shows that 
$$\chi\Pi_\lambda\chi_j(1-\tilde{\chi}_j)=0=\chi_j(1-\tilde{\chi}_j)\Pi_\lambda\chi.$$
We also have $\Tr^\flat\chi\Pi_\lambda\chi_j\tilde{\chi}_j=\Tr^\flat\chi_j\tilde{\chi}_j\Pi_\lambda\chi=0$ since $\supp\chi\cap\supp\chi_j\tilde{\chi}_j=\varnothing$.
\end{itemize}
We conclude that
\begin{align*}
    \dim\Res_X(\lambda)=\Tr(\Pi_\lambda)=\sum\limits_j \Tr(\chi_j\Pi_{j,\lambda}\chi_j).
\end{align*}
Here by \cite[(0.14)]{open}, the Schwartz kernel of $\Pi_{j,\lambda}$ is supported inside $\Gamma_{+,j}\times\Gamma_{-,j}$ where $\Gamma_{\pm,j}$ are the outgoing/incoming tail for $(\overline{\Ucal}_j,X_j)$. Since $K_j=\Gamma_{+,j}\cap\Gamma_{-,j}$, we see that
$$\Tr(\chi_j\Pi_{j,\lambda}\chi_j)=\Tr^\flat(\chi_j\Pi_{j,\lambda}\chi_j)=\Tr^\flat(\Pi_{j,\lambda})=\Tr(\Pi_{j,\lambda})=\dim\Res_{X_j}(\lambda)$$
and this concludes the proof. 
\end{proof}

\begin{rem}
Meddane gives an independent proof of Proposition \ref{p:res-axioma} in his PhD thesis \cite{meddanethesis}.
\end{rem}

\section{Further discussions}
\label{s:example}

\subsection{Generalization to actions on vector bundles}
\label{s:bundle}

In \cite{open}, the authors also studied the case of vector bundles. Consequently, our Theorem \ref{t:upper} and Theorem \ref{t:localtrace} works in the case of vector bundles with minor modifications of the proof. But Theorem \ref{t:lower} may not be true for vector bundles, since our proof relies on the positivity of the geometric side of the trace formula, which generally fails for vector bundles. In fact, \cite{jezequel} constructed an example of an Axiom A diffeomorphism with a complex weight such that there is no resonance.

Let us state the setup.
\begin{defi}
Let $(\Ucal,X)$ be an open hyperbolic system as in Section \ref{s:open}. Suppose there is a vector bundle $\Ecal$ over $\overline{\Ucal}$ and a first order differential operator $\mathbf X$ such that
\begin{align}\label{vb_X_lift}
    \Xbf (f\ubf)= (Xf)\ubf+f(\Xbf\ubf),\quad f\in C^\infty(\overline{\Ucal}), \ubf\in C^\infty(\overline{\Ucal},
    \Ecal),
\end{align}
then we call such a triple $(\Ucal,\Ecal,\Xbf)$ an open hyperbolic system with vector bundle $\Ecal$.
\end{defi}
One fixes an arbitrary extension of $\Ecal$ and $\Xbf$ to $\Mcal$ so that \eqref{vb_X_lift} holds. \cite{open} proved 
\begin{align}
    \mathbbm{1}_\Ucal(\Xbf+\lambda)^{-1}\mathbbm{1}_\Ucal:C_c^\infty(\Ucal;\Ecal)\to \Dcal'(\Ucal;\Ecal),\quad\Re\lambda\gg 1
\end{align}
admits a meromorphic continuation to $\lambda\in\CC$. Our proof in Section \ref{s:upper} still works with pseudodifferential operators on $\Ecal$. So we have for any constant $\beta>0$, 
\begin{align}
    \# \Res(-i\Xbf)\cap\{\lambda\in\mathbb{C}\,:\, |\Re\lambda-E|\leq 1, {\rm Im}\, \lambda\geq -\beta\}= \mathcal{O}(E^n).
\end{align}

Moreover, let the parallel transport
\[ \alpha_{x,t}:\Ecal(x)\to \Ecal(\varphi^t(x))\]
be defined by $\alpha_{x,t}(\ubf(x))=e^{-t\Xbf}\ubf(\varphi^t(x))$. As explained in \cite[Section 4.1]{open}, $\alpha_{x,t}$ only depends on the value of $\ubf$ at $x$ and is thus well-defined. For a closed trajectory $\gamma$ (not necessarily primitive), we denote
$\Tr\alpha_\gamma:=\Tr \alpha_{\gamma(t),T}$
where $T$ is the period of $\gamma$. Note $\Tr\alpha_\gamma$ is well-defined since $\Tr \alpha_{\gamma(t),T}$ is independent of $t$.

We recall the Atiyah--Bott--Guillemin trace formula from \cite[(4.6)]{open}, which states
\begin{align}
    \Tr^\flat\int_0^\infty \psi(t)\chi e^{-t\Xbf}\chi dt = \sum\limits_\gamma \frac{\psi(T_\gamma)T_\gamma^\#\Tr\alpha_\gamma}{|\det(I-\Pcal_\gamma)|},\quad\psi\in C_c^\infty(0,\infty),
\end{align}
where $\chi\in C_c^\infty(\Ucal)$ is any function such that $\chi=1$ near $K$.
The proof in Section \ref{s:lower} except Subsection \ref{s:weaklowerbound} still works and gives for any $A>0$,
\begin{align}\label{vb_localtrace}
    \sum\limits_{\mu\in\Res(P),\Im(\mu)>-A}e^{-i\mu t}+F_A(t)=\sum\limits_{\gamma }\frac{T_\gamma^\#\delta(t-T_\gamma)\Tr\alpha_\gamma}{|\det(I-\mathcal{P}_\gamma)|},\quad t>0
\end{align}
in $\mathcal{D}'((0,\infty))$, where the sum on the right hand side is taken over all closed orbits, $\mathcal{P}_\gamma$ is the Poincar\'e map, and $F_A\in \Scal'(\RR)$ is a distribution supported in $[0,\infty)$ such that
\begin{align*}
    |\widehat{F}_A(\lambda)|=\Ocal_{A,\epsilon}(\langle \lambda\rangle^{2n+1+\epsilon}),\quad \Im \lambda<A-\epsilon
\end{align*}
for any $\epsilon>0$. However, since $\Tr\alpha_\gamma$ may not be positive, we cannot conclude a lower bound as in Theorem \ref{t:lower} from \eqref{vb_localtrace}.

\subsection{Suspensions of Axiom A maps}
As in \cite[Appendix B]{local} by Naud, we have the following corollary about the resonances for suspension of Axiom A maps.

\begin{prop}\label{p:maptoflow}
Suppose we have an Axiom A map $\varphi:M\to M$ with a lift to a smooth vector bundle $E\to M$, i.e. there is a lift $\tilde{\varphi}:E\to E$ linear on the fiber such that the following diagram commutes.
\begin{align*}
    \xymatrix{E\ar[r]^{\tilde{\varphi}}\ar[d]&E\ar[d]\\
              M\ar[r]^{\varphi}&M}
\end{align*}
 We can naturally define an Axiom A flow system on the suspension $\Ecal\to \Mcal$ (see \S \ref{s:bundle}) where
\begin{align*}
    \Ecal=E\times[0,1]/(e,0)\sim(\tilde{\varphi}(e),1),\quad \Mcal=M\times[0,1]/(m,0)\sim (\varphi(m),1).
\end{align*}
The flow $\varphi^t:\Mcal\to \Mcal$ and lift $\tilde{\varphi}^t:\Ecal\to\Ecal$ are defined by translations.

Let $e^{-i\lambda_k},k\in\mathbb{N}$ be the resonances of $\tilde{\varphi}$. Then the suspension flow $\tilde{\varphi}^t$ has resonances given by
\begin{align*}
    \bigcup_{k\in\mathbb{N}}\lambda_k + 2\pi \ZZ.
\end{align*}
\end{prop}
\begin{proof}
This comes from comparing the trace formula for Axiom A maps and Axiom A flows.
The local trace formula for Axiom A maps (see e.g. \cite[Theorem 2.8]{jezequel}) says
\begin{align}
\label{map_trace}
    \sum\limits_{\varphi^p(x)=x}\frac{\Tr\alpha_{x,p}}{|\det(I-d(\varphi^{-p}))|}=\sum\limits_{\Im\lambda_k>-A}e^{-ip\lambda_k}+\Ocal(e^{-A p}).
\end{align}
The local trace formula for Axiom A flows \eqref{vb_localtrace} says that
\begin{align}
\label{flow_trace}
    \sum\limits_{\gamma }\frac{T_\gamma^\#\delta(t-T_\gamma)\Tr\alpha_{\gamma}}{|\det(I-\mathcal{P}_\gamma)|}=\sum\limits_{\mu\in\Res(P),\Im(\mu)>-A}e^{-i\mu t}+F_A(t).
\end{align}
By Poisson summation formula, as a distribution of $t\in (0,\infty)$,
$$\sum_{p=1}^\infty e^{-ip\lambda_k}\delta(t-p)
=e^{-it\lambda_k}\sum_{p=1}^\infty\delta(t-p)
=e^{-it\lambda_k}\sum_{j=-\infty}^\infty e^{-i2\pi jt}.$$
Take the sum from $p=1$ to $\infty$ in \eqref{map_trace}, we get as a distribution of $t\in (0,\infty)$,
\begin{align*}
     \sum\limits_{\gamma }\frac{T_\gamma^\#\delta(t-T_\gamma)\Tr\alpha_{\gamma}}{|\det(I-\mathcal{P}_\gamma)|}&=\sum\limits_{p=1}^\infty \left(\sum\limits_{\Im\lambda_k>-A}e^{-ip\lambda_k}+\Ocal(e^{-A p})\right)\delta(t-p)\\
     &= \sum\limits_{\Im\lambda_k>-A}\sum\limits_{j=-\infty}^\infty e^{-i(\lambda_k+2\pi j) t}+\tilde{F}_A(t)
\end{align*}
where
\begin{align*}
    \Fcal(\tilde{F}_A)(\lambda)=\Ocal_{A,\epsilon}(1),\quad \Im\lambda<A-\epsilon
\end{align*}
for any $\epsilon>0$. Compare it with \eqref{flow_trace}, we get
$\Res(P)=\{\lambda_k+2\pi j:j\in\ZZ\}$.
\end{proof}

\begin{rem}
For the suspension flow, our lower bound \eqref{e:axioma-lower} is almost optimal: 
\begin{equation}
\# \Res(P)\cap\{\lambda\in\mathbb{C}\,:\, |\lambda|<E, \Im\lambda>-\beta\}=c_\beta E+\mathcal{O}(1)
\end{equation}
with $c_\beta>0$ when $\beta$ is large enough. Therefore we cannot take $\delta>1$ in \eqref{e:axioma-lower}. We conjecture that \eqref{e:axioma-lower} actually holds for $\delta=1$.
\end{rem}

A particular example is an Axiom A map $\varphi:M\to M$ with a smooth weight function $g:M\to \CC$. Let $E=M\times \CC$ be the trivial bundle and the lift be 
\begin{align*}
    \tilde{\varphi}(m,e)=(\varphi(m),g(m)e):E\to E.
\end{align*}

As an application, we apply Proposition \ref{p:maptoflow} to the following example of Axiom A maps given in
\cite{jezequel}.
\begin{prop}
\label{p:jezequel}
There exists a smooth diffeomorphism $T$ of $S^4$, a hyperbolic basic set $K$ of $T$, and a smooth weight function $g:S^4\to \RR$ strcitly positive on $K$ and an ordering $\{e^{-i\lambda_k}\}$ of resonances with $\Im\lambda_1\geq \Im\lambda_2\geq \cdots$
such that for all $p\geq 1$,
$\sum\limits_{k=1}^\infty e^{-ip\lambda_k}$ does not converge.
\end{prop}
As a consequence, its suspension which is an Axiom A flow on a smooth $5$-dimension manifold, with a smooth weight function, the convergence of the sum
\begin{align*}
    \sum\limits_{\mu\in\Res(P)}e^{-i\mu t}
\end{align*}
in $\Dcal'((0,\infty))$ is not clear: if we write the sum as
$$\sum_{k=1}^\infty\sum_{j\in\ZZ}e^{-i(\lambda_k+2\pi j)t},$$
by Poisson summation formula, we get
$$\sum_{k=1}^\infty\sum_{p=1}^\infty e^{-i\lambda_kt}\delta(t-p)
=\sum_{k=1}^\infty\sum_{p=1}^\infty e^{-i\lambda_kp}\delta(t-p)$$
which is divergent by Proposition \ref{p:jezequel}. This suggests that a global version of the trace formula \eqref{e:localtrace}:
$$\sum\limits_{\mu\in\Res(P)}e^{-i\mu t}=\sum_{\gamma}\frac{T_\gamma^\#\delta(t-T_\gamma)}{|\det(I-\mathcal{P}_\gamma)|},\quad t>0$$
may not be well-defined for general $C^\infty$ manifolds if the left-hand side is summed in a wrong way. On the other hand, Jézéquel \cite{globaltrace} proves a global trace formula for ``ultra-differentiable'' Anosov flows which includes any Gevrey classes.

\subsection{Morse--Smale flows}
\label{s:morsesmale}
Dang--Rivi\`ere \cite{morsesmaleg, morsesmale1, morsesmale2} studied resonances for Morse--Smale flows. We use our trace formula to give a direct computation for the resonances here, suggested by Malo J\'{e}z\'{e}quel. This is also discussed in Meddane's PhD thesis \cite{meddanethesis}.

\begin{prop}
\label{p:res-morsesmale}
\begin{itemize}
    \item Suppose we have an open system $(\Ucal,\Ecal,\Xbf)$ with a single closed orbit $\gamma$ of period $T$. Let $\epsilon_\gamma=0$ if $E_s|_{\gamma}$ is orientable and $\epsilon_\gamma=1$ if $E_s|_{\gamma}$ is not orientable. Suppose $s=\rank E_s|_\gamma$ and $\mathcal{P}_\gamma$ has eigenvalues $e^{-\lambda_1},\ldots,e^{-\lambda_{d}}$ such that $\Re\lambda_1\leq \Re\lambda_2\leq\cdots\leq \Re\lambda_s<0<\Re\lambda_{s+1}\leq \cdots\leq \Re\lambda_d$ and $\alpha_\gamma$ has eigenvalues
$e^{-\mu_1},\ldots,e^{-\mu_r}$. Then the resonances (counted with multiplicity) are given by
\begin{equation}
\label{e:res_closedorbit}
\begin{split}
    \Res(P) = \Big\{(-i\mu_l + \sum\limits_{j=1}^s ik_j\lambda_j-\sum\limits_{j=s+1}^d ik_j\lambda_j&\;+2\pi n +\pi \epsilon_\gamma)/T: n\in\mathbb{Z}, l=1,\ldots,r\\
   &\; k_1,\ldots,k_s\geq 1,k_{s+1},\ldots, k_d\geq 0\Big\}.
   \end{split}
\end{equation}
\item  Suppose we have an open system $(\Ucal,\Ecal,\Xbf)$ with a single hyperbolic fixed point $x_0$. Suppose $d\varphi^t|_{x_0}=e^{tA}$ and $A$ has eigenvalues $\lambda_1,\ldots,\lambda_d$ with $\Re\lambda_1\leq \Re\lambda_2\leq\cdots\leq \Re\lambda_s<0<\Re\lambda_{s+1}\leq\cdots\leq\Re\lambda_d$ where $s=\dim E_s|_{x_0}$. Let $e^{-\mu_1 t},\ldots,e^{-\mu_r t}$ be the eigenvalues of $\alpha_{x_0,t}$, then the resonances (counted with multiplicity) are given by
\begin{equation}
\label{e:res_fixedpoint}
\begin{split}
    \Res(P)= \{i(-\mu_l+\sum_{j=1}^sk_j\lambda_j-\sum_{j=s+1}^dk_j\lambda_j):
    &\;k_1,\ldots,k_s\geq 1;k_{s+1},\ldots,k_d\geq 0\\
    &\;l=1,\ldots,r\}.
\end{split}
\end{equation}
\end{itemize}

\end{prop}
For general Morse--Smale flow, we can use Proposition \ref{p:res-axioma} to obtain that the resonances are exactly the union of \eqref{e:res_closedorbit} for each closed orbit and \eqref{e:res_fixedpoint} for each fixed point.

\subsubsection{A single closed orbit}
In the case of a single closed orbit $\gamma$ of period $T$, we can use the trace formula \eqref{vb_localtrace}:
\begin{align}\label{e:tracef_closedorbit}
     \sum\limits_{\mu\in\Res(P),\Im(\mu)>-A}e^{-i\mu t}+F_A(t)=\sum\limits_{n=1}^\infty\frac{T\delta(t-nT)\Tr\alpha_\gamma^n}{|\det(I-\mathcal{P}_\gamma^n)|},\quad t>0.
\end{align}
\begin{itemize}
    \item If $E_s|_{\Pcal_\gamma}$ is orientable, then $|\det(I-\mathcal{P}_\gamma^n)|=(-1)^s\det(I-\mathcal{P}_\gamma^n)$. This is because 
    \begin{align*}
        \det(I-\mathcal{P}_\gamma^n|_{E_s})>0,\quad (-1)^s\det(I-\mathcal{P}_\gamma^n|_{E_s})=\det(I-\mathcal{P}_\gamma^{-n}|_{E_s})\det(\mathcal{P}_\gamma^n|_{E_s})>0.
    \end{align*}
    \item If $E_s|_{\Pcal_\gamma}$ is orientable, then $|\det(I-\mathcal{P}_\gamma^n)|=(-1)^{s+n}\det(I-\mathcal{P}_\gamma^n)$. This is because
    \begin{align*}
        \det(I-\mathcal{P}_\gamma^n|_{E_s})>0,\quad (-1)^{s+n}\det(I-\mathcal{P}_\gamma^n|_{E_s})=\det(I-\mathcal{P}_\gamma^{-n}|_{E_s})(-1)^n\det(\mathcal{P}_\gamma^n|_{E_s})>0.
    \end{align*}
\end{itemize}
Thus we conclude
$|\det(I-\mathcal{P}_\gamma^n)|=(-1)^{s+n\epsilon_\gamma}\det(I-\mathcal{P}_\gamma^n)$.

The trace formula \eqref{e:tracef_closedorbit} then gives
\begin{align*}
    &\sum\limits_{\mu\in\Res(P),\Im(\mu)>-A}e^{-i\mu t}+F_A(t)=\sum\limits_{n=1}^\infty\frac{T\delta(t-nT)(e^{-n\mu_1}+\cdots+e^{-n\mu_r})}{(-1)^{s+n\epsilon_\gamma}\prod_j(1-e^{-n \lambda_j})}\\
    &=T\sum\limits_{n=1}^\infty\sum\limits_{j=1}^r\delta(t-nT) (-1)^{n\epsilon_\gamma}   e^{-n\mu_j}e^{n\lambda_1}\cdots e^{n\lambda_s}\sum\limits_{k\in\mathbbm{N}^d}e^{n k_1\lambda_1}\cdots e^{nk_s \lambda_s}e^{-nk_{s+1}\lambda_{s+1}}\cdots e^{-n k_d\lambda_d}\\
    &=\sum\limits_{n\in\ZZ,j} e^{-2\pi i n t/T-\pi i\epsilon_\gamma t/T}e^{-\mu_j t/T}e^{(\lambda_1 +\cdots+\lambda_s) t/T}\sum\limits_{k\in\mathbbm{N}^d}e^{(k_1\lambda_1+\cdots+k_s\lambda_s-k_{s+1}\lambda_{s+1}-\cdots-k_d\lambda_d)t/T}.
\end{align*}
Thus we conclude \eqref{e:res_closedorbit}.

\subsubsection{A hyperbolic fixed point}
Now we turn to the case with a single hyperbolic fixed point. Consider the operator $\chi e^{-t\Xbf}\chi$: the Schwartz kernel is 
\begin{align*}
    \chi(x)\alpha_{y,t}\delta(\varphi^{-t}(x)-y)\chi(y)
\end{align*} 
and its restriction to the diagonal is given by
$\alpha_{x_0,t}\delta_{x_0}/|\det(I-d_x\varphi^{-t}(x_0))|$. Similar to above, we have
\begin{align*}
    |\det(I-d_x\varphi^{-t}(x_0))|=(-1)^s\det(I-d_x\varphi^{-t}(x_0)).
\end{align*}
Thus
\begin{align*}
    \Tr^\flat(\chi e^{-t\Xbf}\chi)=\Tr\alpha_{x_0,t}/|\det(I-d\varphi^{-t}|_{x_0})|=(-1)^s\left(\sum_l e^{-\mu_l t}\right)\prod_j(1-e^{-\lambda_jt})^{-1}.
\end{align*}
We integrate in $t$,
\begin{align*}
    &\Tr^\flat\int_0^\infty \chi e^{-t(\Xbf-i\lambda)}\chi dt = \int_0^\infty (-1)^s\left(\sum_l e^{-\mu_l t}\right)\prod_j(1-e^{-\lambda_jt})^{-1} e^{i\lambda t}dt\\
    &=\sum_l \int_0^\infty e^{-\mu_lt +(\lambda_1+\cdots+\lambda_s)t}\sum\limits_{k\in\mathbb{N}^d} e^{(k_1\lambda_1+\cdots k_s\lambda_s)t-(k_{s+1}\lambda_{s+1}+\cdots k_d\lambda_d)t+i\lambda t}dt\\
    &=\sum_l\sum\limits_{k_1,\cdots,k_s=1}^\infty\sum\limits_{k_{s+1},\cdots,k_d=0}^\infty (-i\lambda+\mu_l-k_1\lambda_1-\cdots-k_s\lambda_s+k_{s+1}\lambda_{s+1}+\cdots + k_d\lambda_d)^{-1}.
\end{align*}
So the resonances are given by \eqref{e:res_fixedpoint}.

\subsection{Geodesic flow on hyperbolic manifolds and the fractal Weyl law}
For geodesic flows on closed hyperbolic manifolds, Dyatlov--Faure--Guillarmou \cite{DFG} computed the resonances in terms of the eigenvalues of the Laplace operator. Guillarmou--Hilgert--Weich \cite{GHW} and
Hadfield \cite{hadfield} gave the following generalization to geodesic flows on convex cocompact hyperbolic manifolds.
\begin{prop}(\cite[Theorem 1.1]{hadfield})
Let $M=\Gamma\backslash\mathbb{H}^{n+1}$ be a smooth convex cocompact hyperbolic manifold. Then for $\lambda_0\in\mathbb{C}\setminus \left(-\frac{n}{2}-\frac{1}{2}\mathbb{N}_0\right)i$, we have the following isomorphism between the space of generalized resonance states for the geodesic flow and the space of generalized resonance states for the Laplacian operator on certain divergence free and trace free symmetric tensors:
\begin{align*}
    \Res_X(\lambda_0)\cong \bigoplus\limits_{m\in\mathbb{N}_0}\bigoplus\limits_{0\leq k\leq m/2}\Res_{\Delta,m-2k}(-i\lambda_0+m+n).
\end{align*}
\end{prop}

Suppose the limit set of $M=\Gamma\backslash\mathbb{H}^{n+1}$ has dimension $\delta_\Gamma$, we have the fractal Weyl upper bound for the number of resonances by Datchev--Dyatlov \cite{DDfwl}:
\begin{align*}
    \sum\limits_{|\Re\lambda - E|\leq 1,\Im\lambda>-A} \dim\Res_{\Delta,m-2k}(-i\lambda_0+m+n) =\Ocal( E^{\delta_\Gamma}).
\end{align*}
Thus
\begin{align*}
    \Res(P)\cap\{|\Re\lambda - E|\leq 1,\Im\lambda>-A\}=\Ocal(E^{\delta_\Gamma}).
\end{align*}
Since $\delta_\Gamma\in [0,n)$, this fits into our upper bound ($\dim S^*M=2n+1$) and suggests the following conjecture on the fractal Weyl upper bound on the number of Pollicott--Ruelle resonances:
\begin{conj}
\label{conj}
Suppose we have an open system such that the trapped set in $p^{-1}(1)$, which is $E_0^\ast\cap p^{-1}(1)$, has dimension $\delta$. Then for any $\epsilon>0$, we expect
\begin{equation}
\label{e:fractalweyl}
    \#\Res(P)\cap\{|\Re\lambda - E|\leq 1,\Im\lambda>-A\}=\Ocal(E^{\frac{\delta-1}{2}+\epsilon}).
\end{equation}
Under certain purity conditions, \eqref{e:fractalweyl} should hold without $\epsilon$. (Here since the trapped set may be fractal, one needs to make a suitable choice of the concept of dimension.)
\end{conj}

Several other special cases are known:
\begin{itemize}
\item Datchev--Dyatlov--Zworski \cite{DDZ} proved \eqref{e:fractalweyl} in the case of contact Anosov flows which is also optimal in the case of geodesic flow on closed hyperbolic manifolds.
\item Faure--Tsujii \cite{fwl} proved a version of \eqref{e:fractalweyl} for general Anosov flows with exponents given by the H\"{o}lder regularity of $E_u\oplus E_s$, using an anisotropic version of Minkowski dimension.
\item For suspension flows, Proposition \ref{p:maptoflow} gives a better bound $\Ocal(1)$ than \eqref{e:fractalweyl} in general. In this case, the trapped set in $p^{-1}(1)$ is contained in a Lagrangian submanifold $\{(x,u;du):(x,u)\in\mathcal{M}\}$ here $(x,u)\in M\times[0,1]$ is a local coordinate on $\mathcal{M}=M\times[0,1]/\sim$ and $du$ is the $1$-form obtained by $u$ which is well-defined on $\mathcal{M}$. This allows better localization $\Ocal(h)$ in phase space rather than $\Ocal(h^{1/2})$ for general fractal trapped set.
\item For Morse--Smale flows, Proposition \ref{p:res-morsesmale} gives \eqref{e:fractalweyl}: for closed orbits, $\delta=1$ and by \eqref{e:res_closedorbit},
\begin{align*}
    \#\Res(P)\cap \{|\Re\lambda - E|\leq 1,\Im\lambda>-A\}=\Ocal(1).
\end{align*} 
and for fixed points, the trapped set in $p^{-1}(1)$ is empty and by \eqref{e:res_fixedpoint},
\begin{align*}
    \Res(P)\cap \{|\Re\lambda - E|\leq 1,\Im\lambda>-A\}=\varnothing,\quad E\gg 1.
\end{align*}
Both results match the fractal Weyl upper bound \eqref{e:fractalweyl}.
\end{itemize}



\def\arXiv#1{\href{http://arxiv.org/abs/#1}{arXiv:#1}}

\end{document}